\documentclass[11pt]{article}
\usepackage{amsmath,amssymb,amsthm}
\usepackage{booktabs}
\usepackage{multirow}
\usepackage[ruled,vlined]{algorithm2e}
\usepackage{color}
\usepackage[T1]{fontenc}
\newtheorem{theorem}{Theorem}[section]
\newtheorem{lemma}[theorem]{Lemma}
\newtheorem{proposition}[theorem]{Proposition}
\newtheorem{corollary}[theorem]{Corollary}
\newtheorem{remark}[theorem]{Remark}

\theoremstyle{definition}
\newtheorem{notation}[theorem]{Notation}

\def\ga{{\gamma}}

\def\Om{\Omega}
\def\Si{\Sigma}

\def\d{\delta}

\def\la{\langle}
\def\ra{\rangle}

\newcommand{\R}{\mathbb{R}}
\newcommand{\Z}{\mathbb{Z}}

\newcommand{\Sym}{{\rm Sym}}
\newcommand{\Alt}{{\rm Alt}}
\newcommand{\Prob}{{\rm Prob}}
\newcommand{\Del}{{\Delta}}
\newcommand{\calS}{\mathcal{S}}
\newcommand{\calU}{\mathcal{U}}
\newcommand{\calP}{\mathcal{P}}
\newcommand{\calR}{\mathcal{R}}

\newcommand{\calN}{\mathcal{N}}
\newcommand{\calF}{\mathcal{F}}
\newcommand{\calG}{\mathcal{G}}
\newcommand{\calB}{\mathcal{B}}
\newcommand{\calC}{\mathcal{C}}

\newcommand{\prob}{\mbox{\rm Prob}}
\newcommand{\true}{\mbox{\tt true}}
\newcommand{\false}{\mbox{\tt false}}
\newcommand{\tracecycle}{\mbox{\sc TraceCycle}}

{

}

\begin{document}

\title{Identifying
long cycles in finite alternating and symmetric groups 
acting on subsets}
\author{Steve Linton$^\ast$, 
Alice C.  Niemeyer$^{+}$ and Cheryl E. Praeger$^{++}$}
\maketitle


\begin{abstract}
Let $H$ be  a  permutation group  on a  set $\Lambda$,  which is
permutationally isomorphic to a  finite alternating or symmetric group
$A_n$ or  $S_n$ acting on the $k$-element subsets  of points
from $\{1,\ldots,n\}$,  for some arbitrary  but fixed $k$. Suppose
moreover that no isomorphism with this action is known.  
We show that key elements
of $H$ needed to construct such an isomorphism $\varphi$,
such as those whose image under $\varphi$ is an
$n$-cycle or $(n-1)$-cycle, can  be  recognised with  high
probability by the lengths of  just four of their cycles in $\Lambda$.
\end{abstract}
\noindent
{\bf 2010 AMS Classification:} 
Primary 20B30; Secondary  60C05, 20P05, 05A05\\

\footnote{
${}^\ast$ School of Computer Science, University of St. Andrews, North Haugh,
St. Andrews, Fife, KY16 9SX, Scotland\\
{\tt sal@cs.st-andrews.ac.uk}\\
${}^{+}$ Department of Mathematics and Statistics, Maynooth
University, Co. Kildare, Ireland.\\   
{\tt Alice.Niemeyer@nuim.ie}\\
${}^{++}$ Centre for the Mathematics of Symmetry and Computation, The
University of Western Australia, 35 Stirling Hwy, Crawley, WA 6009,
Australia.\\
{\tt Cheryl.Praeger@uwa.edu.au}}

\section{Introduction}
The  second and  third authors predicted in  \cite{NieP3} that, 
for  a  permutation group  $H$ on a  set $\Lambda$,  which is
permutationally isomorphic to a  symmetric group
 $S_n$ acting on the $k$-element subsets  of points
from $\{1,\ldots,n\}$ (that is in its $k$-set action),  
for some arbitrary  but fixed $k$,  it should be possible to recognise
an element in $H$ corresponding to an $n$-cycle in $S_n$ by the
lengths of just four of its cycles in $\Lambda.$ The purpose of this
paper is to prove this result.

\begin{theorem}\label{the:main-short}  
Let $H$ be 
 a  permutation group  on a  set $\Lambda$,  which is
permutationally isomorphic, via an unknown isomorphism $\varphi$, to a  finite
 symmetric group  $S_n$ in its  $k$-set action, for some $k$.  
Let $h$ be a uniformly distributed random element of $H$ 
and let $\lambda_1, \ldots,\lambda_4$ be independent, 
uniformly distributed random points of $\Lambda.$ 
 Then there exist  positive
constants $N_0$  and $c$ such that, for $n\ge N_0$,
$$\prob \left( 
\begin{array}{l}
\varphi(h)\mbox{\ is\ an \ }\\
n\mbox{-cycle}
\end{array}
\left\vert 
\begin{array}{l} 
\mbox{the\ } h\mbox{-cycle\ containing\ }\lambda_i\mbox{\ has}\\
\mbox{length\ } n,\mbox{\ for\  }i=1,\ldots, 4\ \\
\end{array}\right. 
\right) > 1 - \dfrac{c}{n^{\tfrac{1}{6}}}.$$
\end{theorem}

Subset actions of $S_n$ and the alternating group $A_n$ play a crucial
role  in algorithms  for  permutation groups.   They  are examples  of
`large-base'   actions.   Most   primitive   permutation  groups   are
`small-base' and  very efficient  algorithms are available  to compute
with  them  (for  a  detailed  definition  see  \cite[p.~51]{Seress}).
However these  algorithms become prohibitively  expensive when applied
to  large-base groups  and, therefore,  alternative means  of handling
large-base groups  are essential (see \cite[Chapter  10]{Seress} for a
discussion  on currently  available  algorithms for  this case).   The
large-base primitive  permutation groups  all contain in  their socles
alternating  groups  with  associated  subset actions.  Hence  finding
efficient  algorithms for  these actions  is important.  

The probabilistic algorithm described in \cite{LNP} 
recognises alternating and symmetric groups in their
actions on $k$-sets constructively. It takes as input a 
group $H$ and an integer $n$. Under the assumption that $H$ is
isomorphic to $S_n$, the algorithm examines a number of random
elements of $H$ seeking to find \emph{key-elements}, namely
elements   for which good estimates for their proportions in $S_n$  are known
and which have particular
properties. Should this search fail, the algorithm concludes that the
assumption that $H$ is isomorphic to $S_n$ is incorrect and reports
that $H$ is not isomorphic to $S_n$. 
Otherwise, it attempts to construct an isomorphism from $H$ to $S_n$.
Since this algorithm is randomised, there is a possibility that the
search appears to succeed but in fact has not found suitable
elements. In  most settings this will be detected as part of the
larger algorithm, see \cite{LNP}.

The theoretical  underpinning of the  algorithm described above  is to
determine  very  good  bounds  for  the  probability  of  finding  the
key-elements in  $H$ under  the assumption that  $H$ is  isomorphic to
$S_n$. This  is the purpose of  the current paper. For  the connection
between estimation results and probabilistic algorithms in the context
of recognition algorithms for groups see \cite{NPS}.

The groups  the algorithm can  take as  input need not  be permutation
groups. Rather,  they may belong  to a more general  class of groups
called \emph{groups of black  box permutations} (see \cite{LNP}). This
allows the algorithm to be  employed in different computational models,
for  example,  to deal  with  groups  given  by  a set  of  generating
matrices,  which are  permutation isomorphic  to $S_n$  acting on  the
underlying  vector  space, without  converting  such  a group  into  a
permutation group.

The crucial requirement is the ability to compute the image of a point
under the action of  a generator of the group. Suppose  that $t$ is an
upper bound for the time taken  to perform this action, which might be
viewed as a  black box procedure. Then the time  required to determine
that a  word of length  $r$ in the  group generators permutes  a given
point in a cycle of length $m$ is at most $mrt$. In some contexts this
time can  even be very  much less than the  time required to  find the
product of two group generators.

For suitable $n$ and $k$ the algorithms in \cite{LNP} have running
time (excluding any time needed to read the input) growing
significantly more slowly than $\binom{n}{k}$, implying that they
can use the input black box permutations only through computing their action on
a selection of the $\binom{n}{k}$ points, and not through
examining any other aspect of their structure, or computing any other
elements of the group they generate. For example, checking that the
$n$-th power of an element is the identity may already be more
expensive than our entire algorithm. 

The key elements sought by the algorithm to  recognise $A_n$ or $S_n$ in
its  $k$-set action  are  elements containing  an $m$-cycle  for
large $m$,  as described in Table~1.  Our Theorem~\ref{the:main-short}
follows from a more general theorem, Theorem~\ref{main-theorem}, which
determines the probability  of finding elements of
each of these  types among a certain number of random elements of $H$.
Once these key elements have been found,  a permutational isomorphism
from $H$ on $\Lambda$  to $A_n$  or $S_n$  on  $k$-sets can  be
implemented  using the  methods described in
\cite[Section~4]{BratPak}, especially 
Method~B, or those  described in \cite[Sections~4 and 5]{Bealsetal03},
especially Lemmas~4.1, 4.3 and 5.5.  Alternative methods, focussing in
particular on the $k$-set action, are developed in \cite{LNP}.

\subsection{Context of our results}
In a  seminal collection of  papers, Erd\"os and Tur\'an  initiated the
study of  asymptotic  behaviour of the proportions of  various kinds of 
elements   in   permutation   groups.   For   example,   they   showed
\cite{ETa,ETb}  that for $n$  large enough,  most elements  in the
symmetric group $S_n$ of degree $n$
have order $n^{(\frac{1}{2} +o(1))\log(n)}.$ In the same vein
 Warlimont \cite{Warlimont78}
proved that the  conditional probability that a random  element $g$ in
$S_n$  is  an $n$-cycle,  given  that $g^n  =1$,  is  $1 -  O(n^{-1}).$

Applied algorithmically, Warlimont's result  is used to conclude, from
the fact that the $n$th power  of a `hidden' permutation $g\in S_n$ is
the identity, that  $g$ is almost certainly an  $n$-cycle.  Finding an
$n$-cycle  is  a key  step  in  many algorithms  that  `constructively
recognise'  $S_n$,  so  this  is valuable.  However,  testing  whether
$g^n=1$    requires   $\log(n)$    multiplications   of    black   box
permutations. In  computational models  where a  single multiplication
costs about  $\binom{n}{k}$ operations,  employing this  approach would
not yield  an algorithm  whose running  time grows  significantly more
slowly  than $\binom{n}{k}$.   The results  of \cite{NieP3}  provide a
basis  for extending  this  to a  situation where  we  know only  that
$\left\langle  g\right\rangle$ has  an  orbit of  length  $n$ in  some
action. An  extension of this nature  to $k$-set actions of  $S_n$ and
$A_n$  is   the  subject  of   this  paper.  We  refine   and  improve
significantly the main result of  \cite{NieP3}. For example, we employ
a  similar division  of the  elements of  $S_n$ into  several families
according  to properties  of points  which  lie in  cycles of  lengths
dividing  $n$.  However,  examining  this  subdivision  alone  is  not
sufficient to achieve the results in  this paper. We need to study the
probability that several $k$-element subsets of $\{1,\ldots, n\}$ have
exactly  $n$ distinct  images under  $\langle  g \rangle$  for $g$  an
element  in  one  of  the   families.  Moreover,  in  our  algorithmic
applications we also required analogous  results for elements of $S_n$
and $A_n$ containing $m$-cycles, for $m\geq n-6$.

\bigskip
\begin{table}[ht]
\begin{center}
\begin{tabular}{ccllcrc}
\toprule
Line &G&  $n$&$m$&$r$ & key-elements & $\rho(G,n,m)$ \\ \hline
1 & \multirow{3}{*}{$S_n$}&      & $n$ &  $1$ & $n$-cycle & 1 \\
2 &  &odd & $n-2$ & $2$ & 2-cycle & 1 \\
3 &  &even & $n-3$ & $2$ & 2-cycle & 2/3\\
\midrule
4 & \multirow{6}{*}{$A_n$} &odd   & $n$ & $1$  & $n$-cycle & 1 \\
5 &  &even   & $n-1$ & $1$ & $(n\!\!-\!\!1)$-cycle &$1$\\
6 & & $2$ or $4 \pmod{6}$&$n-3$&$3$& $3$-cycle & $1$\\
7 & & $3$ or $5 \pmod{6}$&$n-4$&$3$ &$3$-cycle& $3/4$\\
8 &  & $0\pmod{6}$       &$n-5$&$3$& $3$-cycle& $7/20$\\
9 &  & $1\pmod{6}$       &$n-6$&$3$& $3$-cycle&$9/40 $\\
\bottomrule
\end{tabular}
\caption{Groups and types of elements}
\label{tbl:elements}
\end{center}
\end{table}

\bigskip

In  Section~\ref{sec:algoapp}  we  briefly  describe  the  algorithmic
application,  and  in  particular   we  explain  the  meaning  of  the
parameters   $r$   and   $\rho$   in   Table~\ref{tbl:elements}.    In
Section~\ref{sec:maintheorem} we introduce the notation which we shall
use throughout  the paper and give  the precise statement  of the main
result      (Theorem~\ref{main-theorem}).       The      proof      of
Theorem~\ref{main-theorem} (and hence of Theorem~\ref{the:main-short})
is given  in Section~\ref{sec:overview}.  In particular  we exhibit an
explicit value for  the constant $c$ of Theorem~\ref{main-theorem}(a)
and Theorem~\ref{the:main-short}.  We present some background material
in                                            Sections~\ref{sec:prelim}
and~\ref{sec:binomial}. Sections~\ref{sec:tracing} - \ref{f-1} contain
the     various    parts    which     are    pulled     together    in
Section~\ref{sec:overview}    for   the    proof    of
Theorem~\ref{main-theorem}.

\section{Algorithmic Application}\label{sec:algoapp}

The results in this paper are motivated by  algorithmic applications
in \cite{LNP} and \cite{N}. 
In these   applications, $H$ is  a permutation group
acting on  a set
$\Lambda$ of $\binom{n}{k}$ points. We wish to test whether $H$ is permutation
isomorphic to $G=A_n$ or $G=S_n$  acting on the set $\binom{\Omega}{k}$ 
of  $k$-element subsets of $\Om =\{1,\dots,n\}$. 
That is to say, whether there is a group isomorphism 
$\varphi:H\rightarrow G$ and a bijection $f:\Lambda\rightarrow \binom{\Omega}{k}$
such that, for each $h\in H$ and $\lambda\in\Lambda$, $(\lambda^h)f =
(\lambda)f^{h\varphi}$.   These isomorphisms will be constructed in
the form of a computer program rather than listing the image of each element.

We say that an element
$h\in H$ {\it corresponds} to an element $g\in G$  if the permutation 
isomorphism $\varphi$ maps $h$ to $g$.
The algorithms
construct a `nice generating' set for $H$ of size 2. 
In the case where $H$ is
permutation isomorphic to $S_n$ 
in its action on $\binom{\Omega}{k}$,
this
generating set consists of elements that, in the natural representation of $S_n$ on $n$ points, correspond 
to an $n$-cycle and a 2-cycle 
interchanging two consecutive points of the $n$-cycle. In the case
where
$H$ is permutation isomorphic to 
$A_n$ in its action on $\binom{\Omega}{k}$ the nice  generating set consists of elements that in $A_n$
correspond to an $n$-cycle or $(n-1)$-cycle, and to
a $3$-cycle.  

We wish to find
these elements by selecting independent, uniformly distributed random
elements from the group $H$. However, the proportion of 2-cycles in
$S_n$ or 
3-cycles in $A_n$ is 
too small to allow us to find such elements directly by random
selection. Therefore, we seek elements in $H$ which correspond to
permutations containing a 2-cycle or a 3-cycle together with one long cycle of length
$m$, say, 
where $m$ is at least $n-6$  and $m$ is coprime to 2 or 3, respectively.
The algorithms in
\cite{LNP} and \cite{N} seek  elements $h\in H$
which correspond to the kinds of elements  $g$ listed in 
Table~\ref{tbl:elements}, where $H$ is permutation isomorphic to 
$G= S_n$  or $G=A_n$, with  $G, n$ as in the 
second and third columns. 
The fourth column, labelled $m$, lists the
length of the 
$m$-cycle which the element $g$ contains.
The fifth column, labelled $r$, lists an integer between 1 and 3. 
Ultimately we wish to find an element $h$ in $H$ which corresponds to an
element in $G$ with cycle type as recorded in the sixth column. 
This element is constructed as a power of the element $h$.

The first element in the nice generating
set for $H$ corresponds to an element
satisfying the conditions of Line~1, 4, or 5, namely it corresponds to 
 an $n$-cycle or an $(n-1)$-cycle.  
The second nice generator  corresponds to 
a $2$-cycle if $G= S_n$  and is constructed from an element $h\in H$ which
corresponds to  $g$ as in Line~2 or 3. If $G = A_n$, 
the second nice generator
corresponds to 
a $3$-cycle and is  constructed from $h\in H$ corresponding to 
an element $g$ as in Line~6, 7, 8 or 9.
The last column, labelled $\rho(G,n,m)$, records a rational number
such that the proportion of elements $h$ of $H$ which correspond 
to elements of $G$ containing an $m$-cycle and with order dividing $rm$ is 
$\frac{\rho(G,n,m)}{m}$  (see (\ref{eq:star})).

The group $H$ acts on a set $\Lambda$ of size
$|\Lambda| = \binom{n}{k}$, and in the context of the
algorithm $m$, $n$ and $k$ are so large that 
it is `too expensive' to  compute the full cycle structure of elements
of  $H$ in  their action  on $\Lambda$.  Instead we  compute  the cycle
lengths of elements $h\in H$ on a handful of  randomly chosen points of
$\Lambda$, that is to say, we  `trace' these points under the action of
$\langle  h\rangle$. 

In computer experiments in {\sf GAP} \cite{GAP}, we discovered that if
$H$  is permutation  isomorphic  to $G=S_n$  or  $A_n$ on  
$\binom{\Omega}{k}$ then, for $m, r$  as in one of
the  lines of  Table~\ref{tbl:elements},  most elements  of $H$  which
produced cycles of lengths a multiple of $m$ and dividing $rm$, when we
traced each  of four  or five independent  random points  of $\Lambda$,
corresponded to elements of  $G$ containing an $m$-cycle.  This computer
experiment  is  formalised in  procedures  {\sc  FindMCycle} and  {\sc
  TraceCycle}. Our  experimental observation turns  out to be  true in
general,  and is  proved  in Theorem~\ref{main-theorem} for
sufficiently large $n$. Our experiments also suggest that the results
hold for  smaller values of $n$.
For clarity  of
exposition the proofs of Theorem~\ref{main-theorem} 
are written in  terms of the action  of $G$ on
$\binom{\Omega}{k}$.

For $n, m$ and $r$ as in one of the lines of
Table~\ref{tbl:elements}, define 
$\calN(n,m)$ to be  the 
set of all $g\in S_n$ that contain an $m$-cycle and 
$\calN_{good}(G,n,m)$ to be  the 
set of all $g\in \calN(n,m)\cap G$ for which $o(g)$ divides
$rm.$ Note that, for given $G, n, m$, only one of the lines of
Table~\ref{tbl:elements} is satisfied, and hence $r$ is determined by $G, n, m.$
We define $\rho(G,n,m)$ to be the rational number satisfying
\begin{eqnarray}\label{eq:star}
\frac{|\calN_{good}(G,n,m)|}{|G|} = \frac{\rho(G,n,m)}{m}.
\end{eqnarray}
As an example of how to interpret this information, consider Line~3 of Table~\ref{tbl:elements}. 
The proportion of elements $g$ of $S_n$ containing an
$(n-3)$-cycle is $\frac{1}{n-3},$ and 2/3 of these elements contain
also a 2-cycle or three 1-cycles on the remaining 3 points. Thus the
proportion of elements of $S_n$ containing an $(n-3)$-cycle and having
order dividing $2(n-3)$ is $\frac{2/3}{n-3} = \frac{\rho(S_n, n,
  n-3)}{n-3}.$ In order to construct a 2-cycle (the entry in column~6 
for this line), we raise the element $g$ to the $(n-3)^{\rm rd}$ power
producing $x = g^{n-3}$. Since $n-3$ is odd, the element $x$ is 
the identity if $g$ has three fixed points, a 2-cycle if $g$
contains a 2-cycle, or possibly a 3-cycle if $g$
contains a 3-cycle and $3$ does not divide $n$. 
Thus  three quarters of  the elements of 
$\calN_{good}(S_n,n,n-3)$ yield a 2-cycle by powering. The algorithm
{\sc FindMCycle} can therefore easily be incorporated into a Monte
Carlo algorithm to construct a transposition in this case: by
repeating  {\sc FindMCycle}   a number of times we will with high
probability construct a transposition  by powering the output of  {\sc
  FindMCycle}.  The other lines have a similar interpretation for
$\rho(G, n, m).$ 

We now describe the two algorithms. Algorithm~\ref{algo:FindMCycle}
assumes that we have a function {\sc RandomGrpElt} which takes
as input a generating set $Y$ for a group $H$ and returns
independent, uniformly
distributed random elements of $H$.
Algorithm~\ref{algo:TraceCycle} assumes that we have a function
{\sc RandomPoint} which  takes as input a finite set $\Lambda$ and 
returns independent, uniformly distributed  random
points of   $\Lambda$. Note that Algorithm~\ref{algo:FindMCycle}
calls Algorithm~\ref{algo:TraceCycle}  and that we assume that
Algorithm~\ref{algo:TraceCycle} has access to the variables of 
Algorithm~\ref{algo:FindMCycle}.

\begin{algorithm}
\caption{{\sc FindMCycle}$(n,  m, r, H, \Lambda, \varepsilon, M)$}
\label{algo:FindMCycle}
\BlankLine
\KwData{Let $(n, m, r)$ be as in one of the lines of
Table~\ref{tbl:elements}.  
Let $H$ be a permutation group with a generating set $Y$ acting 
on a finite set $\Lambda$. Let  $\varepsilon$ be a real number
with $0 < \varepsilon < 1$ and let  $M$ be an integer with $M \ge 4.$}
\KwResult{An element $h\in H$ or {\tt fail};}

\BlankLine
{\em  This algorithm inspects up to $O(n\log(\varepsilon^{-1}))$
  uniformly distributed independent random elements from $H$
to find one which has
  orbits of length a multiple of $m$ and  dividing $rm$ on each of $M$ randomly
  selected points from $\Lambda$. If such an $h\in H$ is found it returns
  $h$, otherwise it returns {\tt fail}.}\\

\BlankLine
Set $N := \left\lceil 5n\log(\tfrac{2}{\varepsilon})\right\rceil$;\\
\For{ $i= 1,\ldots,N$}{
$h_i$  := {\sc RandomGrpElt}$(Y)$\;
\If {{\sc TraceCycle}$(h_i)=$ {\tt true}}{
\Return{$h_i$}\;
}
}
\Return {\tt fail};
\end{algorithm}

\begin{algorithm}
\caption{{\sc TraceCycle}$(h)$}
\label{algo:TraceCycle}
\BlankLine
\KwData{A permutation $h\in H$;}
\KwResult{A boolean `{\tt true}' or `{\tt false}'}
\BlankLine
{\em  This algorithm tests whether the permutation $h\in H$ has
  orbits of length a multiple of $m$ and  dividing $rm$ on $M$ randomly
  selected points from $\Lambda$. If this is the case it returns
{\tt true}, otherwise it returns {\tt false}.}\\
\BlankLine
\For {$i=1,..,M$}{ $\lambda_i := \mbox{\sc RandomPoint}(\Lambda)$;}
Put $\Gamma=\{\lambda_j\}_{j=1}^M$;\\
\For{$\lambda\in \Gamma$}{
\If{$|\lambda^{\langle g\rangle}| \not= r_0m$ \mbox{\rm for some} $r_0\mid
  r$}{\Return{{\tt false};}} 
}
\Return{{\tt true};} 
\end{algorithm}

{\bf Remark:} 
(a) The number $M$ of random points of $\Lambda$ tested in the algorithm \tracecycle\
is often a bounded constant (as, for example, in
Theorem~\ref{main-theorem}), but in our analysis we allow it to be as
large as $O(n)$, see (\ref{eq:M}). 
 
(b) The algorithm \tracecycle\ performs $O(n)$ image computations to
check whether $|\lambda^{\langle g \rangle}| = r_0m$, for each 
random point $\lambda$. Thus if $\xi_{rp}$,  $\xi_{rge}$, $\nu_{im}$, are
upper bounds for the costs of producing a random point using 
{\sc RandomPoint}, producing a random group element using 
{\sc RandomGrpElt}, and computing the image of a point of $\Lambda$ under an element of $H$, respectively, then 
the cost of {\sc FindMCycle} is 
\[
O(n\log(\varepsilon^{-1})(\xi_{rge}+M\xi_{rp}+Mn\nu_{im})). 
\]
This cost is modest when compared with the cost
$\binom{n}{k}\nu_{im}$ of computing the product of two permutations
of $\Lambda$ (especially when $k=O(n)$) or the cost of directly
computing the order of any element.

Our  main  result Theorem~\ref{main-theorem} shows  that  these simple  and  inexpensive
procedures provide an effective way to find and identify elements of $S_n$ and $A_n$ containing $m$-cycles from their
actions on $k$-element subsets.

\section{Statement of the main theorem and notation}\label{sec:maintheorem}

In order  to state  our main theorem  we introduce  several parameters
that are  used throughout the paper.  Suppose that the  triple $(G, n,
m)$ satisfies one of the lines of Table~$\ref{tbl:elements}$, and note
that  $r$ is determined  by $G,  n, m$.  The integer  $M$ used  in the
algorithm {\sc FindMCycle} is assumed to satisfy

\begin{equation}\label{eq:M}
4\leq M\leq  \log\left(\frac{9}{8}\right) \frac{n-2}{2}.
\end{equation}

Let  $d(x)$ be the  number of positive divisors  of an integer $x$.
By  \cite[pp.  395-396]  {NivenZuckermanetal91},  $d(x)=x^{o(1)}$.  In
fact, for every $\d>0$, there  is a positive constant $c_\d$ such that

\begin{equation}\label{eq:cdelta}
d(x)\leq c_\d x^\d 
\end{equation}
for all $x$. Choose real numbers $\d$ and $s $ 
satisfying  

\begin{equation}\label{eq:ds}
0<\d<  \min\{1-s, \frac{s}{3}, s-\frac{1}{2} \}\ \mbox{and\ } \frac{1}{2} < s < \frac{M-1}{M}.   
\end{equation}
Further let
\begin{equation}\label{eq:ell}
\ell = \min\left\{ M(1-s), 3-2s - 2\delta,
1+s-3\delta, 2s-2\delta\right\}. 
\end{equation}
By (\ref{eq:ds}), all of $M(1-s) > 1$, 
$3-2s-2\delta> 1$,  $1 + s - 3\delta > 1$ and $2s-2\delta >1$ hold.
Hence $\ell > 1.$
Next we define the constant $a_\d$ by 

\begin{equation}\label{eq:adelta}
a_\delta := \frac{5}{4}\left(1 + 3\frac{c_\delta}{150^{s-\delta}} + 
\left(\frac{c_\delta}{150^{s-\delta}}\right)^2\right),
\end{equation}
with $c_\delta$ as in (\ref{eq:cdelta}), and the constant 
$b_{M,\delta,s}$, which we usually abbreviate to $b_M$,
by

\begin{equation}\label{eq:mu}
b_{M} =
 \left( \frac{33}{8}\right)^M + 
72\, a_{\delta} c_{\delta}^2 r^{2s+2\delta} + 
6.24\,a_{\delta} c_{\delta}^3 r^{3\delta} + 
\frac{c_{\delta}^2}{r^{2s-2\delta}} + 
 \left(\frac{31}{r^{1-s}}\right)^M.
\end{equation}
The theorem involves an `error probability' $\varepsilon$, that is, a
real number satisfying $0<\varepsilon<1$. We  assume that
the integer $n$ satisfies the following inequalities:

\begin{equation}\label{eq:n}
n\geq\left\{\begin{array}{l}
       12(rn)^s + 6\\
	(rn)^s\log n\\
	\left(\frac{10b_M}{\varepsilon}\right)^{1/(\ell-1)}.      
            \end{array}\right.
\end{equation}

\begin{theorem}\label{main-theorem} 
Let $(G, n, m)$ be as in one of the lines of
Table~$\ref{tbl:elements}$, and let 
$k$ be a positive integer satisfying $2\le k \le n/2$.
Suppose that $H$ is  a permutation group
permutation isomorphic to $G$ acting on $k$-element subsets of
$\{1,\ldots,n\}$  $($via the unknown isomorphism $\varphi : H \rightarrow G)$.
Then  the following hold

\medskip
\noindent
{\rm (a)}$\,$ Let $h$ be a uniformly distributed random element of $H$ 
and let $\lambda_1, \ldots,\lambda_4$ be independent, 
uniformly distributed random points of $\Lambda.$ 
 Then there exist  positive
constants $N_0$  and $c$ such that, for $n\ge N_0$,
$$\prob \left( 
\begin{array}{l}
\varphi(h)\mbox{\ contains\ an \ }\\
m\mbox{-cycle}
\end{array}
\left\vert 
\begin{array}{l} 
\mbox{for\  }i=1,\ldots, 4\ \mbox{the\ }\\
h\mbox{-cycle\ on\ }\lambda_i\mbox{\ has\ length}\\
r_im\mbox{\ for\ some\ }r_i\ |\ r 
\end{array}\right. 
\right) > 1 - \dfrac{c}{n^{\tfrac{1}{6}}}.$$

\medskip
\noindent
{\rm (b)}$\,$ 
Let $M$ be an integer satisfying {\rm(\ref{eq:M})}, and let $s,\d$ be
real numbers satisfying {\rm(\ref{eq:ds})}, and $\ell$ as in
{\rm(\ref{eq:ell})}. Then {\sc FindMCycle} is a Monte Carlo Algorithm
which, given as input the  permutation group $H$,
an error probability  $\varepsilon >0 $ and the integer $M$, returns
an output $h$ such that, provided $n$ satisfies {\rm(\ref{eq:n})}, 
\begin{enumerate}
\item[(i)] the probability that $h\in H$ and $\varphi(h)$ contains an $m$-cycle is
  at least 
$1-\varepsilon,$
\item[(ii)] the probability that $h\in H$ and $\varphi(h)$ does not contain
an $m$-cycle is at most $\varepsilon/2,$ and
\item[(iii)]  the probability that $h= \mbox{\sc Fail}$ is at most
$\varepsilon/2.$
\end{enumerate}
\end{theorem}

\begin{notation}\label{notation}
For the rest of the paper we assume that $n,m, r$ and
$G$ are as in one of the lines of Table~\ref{tbl:elements},
noting that $r$ is determined by $G, n, m$. Let $M$ be an integer
satisfying  {\rm(\ref{eq:M})}, let $s,\d$ be real numbers satisfying
{\rm(\ref{eq:ds})}, and let $\ell, c_\d, a_\d$ and $b_M$ be as in
(\ref{eq:ell}), (\ref{eq:cdelta}), (\ref{eq:adelta}) and (\ref{eq:mu})
respectively. 

Let $S_n$  act naturally on $\Omega=\{1,2,\dots,n\}$.
Let  $k$ and $k_0$ be positive integers
satisfying $2 \le k\le n/2$, and $1 \le k_0 \le k$.
A   $k_0$-element subset  of $\Om$ is called a
$k_0$-subset.

We use the notation in Table~\ref{tbl:notation} to describe an element
$g\in  S_n$, where $\gamma_0$ is a $k_0$-subset of $\Omega$. Here we identify a cycle of $g$ with the subset of $\Omega$ it permutes.
\end{notation}

\begin{table}[ht]
\begin{center}
\begin{tabular}{ll}
\toprule
$c_{k_0}(\ga_0,g)$       &length of the $g$-cycle  containing 
$\ga_0$ on $k_0$-subsets\\
$s$-small $g$-cycle&$g$-cycle in $\Om$ of length less than $(rn)^s$\\
$s$-large $g$-cycle&$g$-cycle in $\Om$ of length at least $(rn)^s$\\
$\Delta(g)$ & union of $g$-cycles in $\Om$ whose lengths divide $rm$\\
$\Sigma(g)$ & $\Omega \setminus \Delta(g)$\\
$v$         & cardinality of $\Delta(g)$\\
$u$         & cardinality of $\Sigma(g)$\\
  \bottomrule
\end{tabular}\medskip
\caption{Table for Notation~\ref{notation}}
\label{tbl:notation}
\end{center}
\end{table}

We define in Table~\ref{tbl:subsets} several classes of elements in $G$.
We usually omit mentioning $n$ and $m$ in our notation.
For example, we refer to 
$\calN(n,m)$ (defined in Section~\ref{sec:algoapp})  simply as $\calN$ and to 
$\calN_{good}(G,n,m)$   simply as $\calN_{good}.$

\begin{table}[ht]
\begin{center}
{\openup 3pt
\begin{tabular}{ll}
\toprule
$\calN$&set of all $g\in S_n$ that contain an $m$-cycle\\
$\calN_{good}$ &
set of all $g\in \calN\cap G$ for which $o(g)$ divides $rm$ \\
$\calF$&set of all $g\in G\setminus \calN$ such that $m \mid o(g)$ \\
$\calR$&set of all $g\in \mathcal{F}$ such that $|\Del(g)|\leq 4{(rn)}^s$ \\
$\calS_0$&  set of all $g\in\mathcal{F}$ such that $|\Del(g)|>4{(rn)}^s$ and\\
& all $g$-cycles in $\Del(g)$ are $s$-small\\
$\calS_1^+$&  set of all $g\in\mathcal{F}$ such that
$|\Delta(g)| > 4{(rn)}^s$, exactly one\\
& $g$-cycle $C$ in $\Del(g)$ is $s$-large, and $|\Delta(g)\setminus C| >3 {(rn)}^s$\\
$\calS_1^-$&  set of all $g\in\mathcal{F}$ such that
$|\Delta(g)| > 4{(rn)}^s$, exactly one\\
&$g$-cycle $C$ in $\Del(g)$ is $s$-large, and $|\Delta(g)\setminus C|
\le 3 {(rn)}^s$\\
$\calS_{\ge 2}$&set of all $g\in\mathcal{F}$ such that $|\Del(g)|>4{(rn)}^s$\\
& and at least two $g$-cycles in $\Del(g)$ are $s$-large\\
\bottomrule
\end{tabular}
}
\medskip
\caption{Families of Elements}
\label{tbl:subsets}
\end{center}
\end{table}

\begin{remark}\label{rem:params}{\rm 
(a) The definition of $a_\d$ is not too critical. We simply need 
$a_\delta$ to be greater than or equal to the right hand side of 
(\ref{eq:adelta}) for the values of
$rm$ we are considering, see Remark~\ref{rem:p0} and Lemma~\ref{lem:p0}. For example, if $rm\geq c_\d^{1/(s-\d)}$ then
we may take $a_\d=25/4$.

(b) Currently Equation~(\ref{eq:n}) limits the practical applicability of
Theorem~\ref{main-theorem} severely, but we note that in our analysis we 
allow $k$ to be as large as $n/2$.  
The first two inequalities of~(\ref{eq:n}) imposed on $n$ are due to the
subdivision of the set of permutations of order divisible by
$m$ into disjoint subsets which depend on $s$. We give a uniform proof
that holds for all values of $k$ in the range $2 \le k \ne n/2$. If,
for example, $k$ were bounded as $n$ increases, then several of the
arguments would be simpler and the constraints on $n$ correspondingly
less severe. 

(c) The main constraint forcing $n$ to be very large is the third
inequality in  (\ref{eq:n}).
For example, for our parameter
choice in Theorem~\ref{main-theorem}, namely $M=4, s = \frac{17}{24}$ and
$\delta = \frac{1}{6}$, we 
have $c_\delta \le 138.32$ and, for $n$ large enough, $a_\delta =
\frac{25}{4}.$
In this case we find $b_M > 2\cdot 10^{8}$
and the last inequality of~(\ref{eq:n})  dictates
$n >  3.3\cdot 10^{112}/\varepsilon^{12}$.
Moreover, even though a larger value of $M$ allows
us to choose a smaller value for 
$c_{\delta}$, the choice might result in a smaller value for
$\ell$,
which in turn has  undesired consequences, making $b_M$ larger, and hence 
requiring $n$ to be larger.
}
\end{remark}

\section{Proof of the Main Theorem}
\label{sec:overview}

The proof  of the main theorem,  Theorem~\ref{main-theorem}, relies on
many supporting results.  In this section we subdivide  the proof into
various parts  and show how these  parts are then  brought together to
give a complete proof. The individual parts of the proof are proved in
later sections.  The main idea of  the proof is to divide the elements
of  $S_n$ that  could possibly  be returned  by {\sc  FindMCycle} into
disjoint families,  and to compute the  probability that $\tracecycle$
returns \true\ for an element of each of these families.  The families
of     elements    in    this     subdivision    are     defined    in
Table~\ref{tbl:subsets},  namely  $\calN,  \calR, \calS_0,  \calS_1^+,
\calS_1^-,\calS_{\geq2}$, and  we use the notation  introduced in this
table throughout the paper.

\begin{proof}[Proof of Theorem~$\ref{main-theorem}(b)$]
We prove this theorem
by analysing the algorithm {\sc FindMCycle}. Let $N = \left\lceil
5n\, \log(\tfrac{2}{\varepsilon})\right\rceil.$ 
 A call to 
algorithm {\sc FindMCycle} can terminate in one of
three possible ways:

\begin{enumerate}
\item[(${\mathcal G}$)] For some $i$ with $1 \le i \le N$ the $i$-th iteration of
  the {\bf for}-loop 
returns an element in $\calN.$
We call this a \emph{good} outcome.
\item[(${\mathcal B}$)] For some $i$ with $1 \le i \le N$ the $i$-th
  iteration of   the {\bf for}-loop 
returns an element which is  not in $\calN.$
We call this a \emph{bad} outcome.
\item[(${\mathcal U}$)] The {\bf for}-loop is executed $N$ times and {\sc
  TraceCycle} returns  \false\ for each of the selected random
  elements. In this case the algorithm returns 
{\sc Fail}.
We call this an \emph{ugly} outcome.
\end{enumerate}

Thus to prove the three parts of Theorem~\ref{main-theorem} we must prove 
\[
\Prob( \calG ) \ge 1 - \varepsilon,\quad \Prob( \calB ) \le \varepsilon/2,\quad
\Prob( \calU ) \le \varepsilon/2. 
\]
Clearly any two of these inequalities implies the third. We shall therefore prove only
$\Prob( \calB ) \le \varepsilon/2$ and 
$\Prob( \calU ) \le \varepsilon/2.$
To study these outcomes more closely we define the following events.

\bigskip

\begin{tabular}{ll}

$E_i$ & the $i$-th iteration of the {\bf for}-loop is
executed. Let $g_i$ denote the\\
&  random element selected in the
$i$-th iteration.\\
$G_i$ & event $E_i$ occurs, $g_i\in \calN$ 
 and {\sc TraceCycle}$(g_i)$ = \true\\ 
$B_i$ & event $E_i$ occurs, $g_i\notin \calN$ 
 and {\sc TraceCycle}$(g_i)$ = \true\\
$U_i$ & event $E_i$  occurs
 and {\sc TraceCycle}$(g_i)$ = \false\\
\end{tabular}
\bigskip

Note that $E_i = G_i \dot\cup B_i \dot\cup U_i$ and that $\prob(E_1) = 1$.
Further, for $i > 1$ we have that
\begin{eqnarray}
E_i = U_1 \cap \ldots \cap U_{i-1} = U_{i-1}. \label{eq:Ei}
\end{eqnarray}
Thus
\begin{equation}\label{gbu}
\begin{array}{lll}
\calG &=& G_1 \vee G_2 \vee \ldots \vee G_N \\
\calB &=& B_1 \vee B_2 \vee \ldots \vee B_N\\
\calU &=& U_1 \wedge U_2 \wedge \ldots \wedge U_N = U_N.\\
\end{array}
\end{equation}

\medskip\noindent
\emph{Proof that $\prob(\calU)\le\varepsilon/2$:}\quad  
For a uniformly distributed random element $g\in G$, let
\[
\begin{array}{lll}
p_1 &=& \prob(\tracecycle(g) = \false \mid g \in \calN_{good})\\
p_2 &=& \prob(\tracecycle(g) = \false \mid g \notin \calN_{good})\\ 
\end{array}
\]
and let  $p =  \frac{\rho}{m}p_1 + \frac{m-\rho}{m} p_2$, where 
$\rho := \rho(G,n,m)$ (see Table~\ref{tbl:elements}), the proportion
of elements of $G$ containing an $m$-cycle that have order dividing $rm$.
Note that, since the proportion of elements containing an $m$-cycle in $S_n$ is
$1/m$, we have $\prob(g\in\calN_{good}) =\frac{\rho}{m}$.

Given $E_i$, the event $U_i$ is the
disjoint union of the events $U_{i1}$, that $g_i \in
\calN_{good}$ and 
$\tracecycle(g_i) = \false$, and $U_{i2}$, that $g_i\not\in
\calN_{good}$ and
$\tracecycle(g_i) = \false$. 
Thus
\begin{eqnarray*}
\prob(U_i \mid E_i) &=& \frac{\rho}{m}\prob(\tracecycle(g_i) = \false 
\mid g_i \in \calN_{good}) \\ 
 &+  &
\frac{m-\rho}{m}\prob(\tracecycle(g_i) = \false \mid g_i \notin \calN_{good})\\
&=&\frac{\rho}{m}p_1 + \frac{m-\rho}{m} p_2=p.
\end{eqnarray*}
Note, in
particular, that this probability is independent of $i$.
By (\ref{eq:Ei}) we have $E_i = U_{i-1}$, and hence
$\prob(U_i) = \prob( E_i) \prob(U_i \mid E_i) = \prob(U_{i-1})\cdot p.$ As this is true for
all $i$ with $1 \le i \le N$, we have 

\begin{eqnarray}
\prob(U_i)&=&p^i,\label{eq:pUi} 
\end{eqnarray}
and in particular, 
$$\prob(\calU) = \prob(U_N) = p^N.$$

The required inequality $\prob(\calU) \le \varepsilon/2$ holds whenever
$p^N \le \varepsilon/2.$ We now prove the latter inequality. 
By Proposition~\ref{prop:mcyc} we have
$1-p_1 \ge \left(\frac{n-2}{n}\right)^M.$ Therefore,
\begin{eqnarray} 
 p & \le &   \frac{\rho}{m}p_1 + \frac{m-\rho}{m} = 1 - \frac{\rho}{m}(1-p_1)
 \le  1 - 
\frac{\rho}{n}
\left(\frac{n-2}{n}\right)^M .
\label{eq:p}
\end{eqnarray}
Now $N= \left\lceil 5 n \log( \tfrac{2}{\varepsilon}) 
\right\rceil = \left\lceil \frac{\log( (\varepsilon/2)^{-1})}{(5n)^{-1}} 
\right\rceil$, and so by Lemma~\ref{lem:eps}, $(1-\frac{1}{5n})^N\leq \varepsilon/2$.
Thus $p^N \le \varepsilon/2$ holds if 
$1 -\frac{\rho}{n}\left(\frac{n-2}{n}\right)^M\leq 1-\frac{1}{5n}$, or equivalently,
if $\left(\frac{n}{n-2}\right)^M\leq 5\rho$. Since $\rho \ge 9/40$ (see Table~\ref{tbl:elements}), it is sufficient to prove that 
$\left(\frac{n}{n-2}\right)^M\leq \frac{9}{8}$.
By our assumption, $M\leq \log(\frac{9}{8})\,\frac{n-2}{2}$, and hence
\[
M\log\left(\frac{n}{n-2}\right) = M\log\left( 1+\frac{2}{n-2}\right) \leq M\,\frac{2}{n-2}\leq \log\left(\frac{9}{8}\right)
\]
and exponentiating both sides gives the required inequality. Thus  
$p^N \le \varepsilon/2$ and hence $\prob(\calU) \le \varepsilon/2$ is proved.

\medskip\noindent
\emph{Proof that $\prob(\calB)\le\varepsilon/2$:}\quad  
Recall the definition of $\calB$ in (\ref{gbu}).
Note that, if $\tracecycle(g) = \true$, then $o(g)$ is divisible by $m$.
Thus, by the definition of
$\mathcal{F}$, for a uniformly distributed, random element $g\in G$,
\begin{eqnarray}\label{eq:q}
q &:=& \prob( g\in \calF \mbox{\ and\ } \tracecycle(g) = \true)\\ \nonumber
&=& \prob( g\not\in \calN \mbox{\ and\ } \tracecycle(g) = \true).
\end{eqnarray}
Now, for all $i$ with $1 \le i \le N$, we have that 
\[
\prob(B_i\mid E_i) = \prob(
g_i\not\in \calN \mbox{\ and\ } \tracecycle(g_i) = \true) = q. 
\]
Hence $\prob(B_i) = \prob(E_i) \prob(B_i \mid E_i ) =
\prob(E_i)\, q.$
If $i \ge 2$ then $E_i = U_{i-1}$ by (\ref{eq:Ei}), and so by (\ref{eq:pUi}), 
$
\prob(B_i) = p^{i-1} q.$
Therefore, 
\begin{eqnarray}\label{eq:probB}
\prob(\calB) &=& \sum_{i=1}^N \prob(B_i) 
= 
q \sum_{i=1}^N p^{i-1}\nonumber \\
&=& 
q
\frac{1-p^N}{1-p} < \frac{q}{1-p}.
\end{eqnarray}

The most substantial part of the paper is devoted to finding an upper
bound for $q$.
It follows from Table~\ref{tbl:subsets} that
\begin{eqnarray*}
\mathcal{F} &=&
\calR\,\dot\cup\,\calS_0\,\dot\cup\,\calS_1^+\,\dot\cup\,\calS_1^-\,\dot\cup\,\calS_{\ge 
2}.
\end{eqnarray*}
Hence
$$q = q(\calR) + q(\calS_0) + q(\calS_1^+) + q(\calS_1^-) + q(\calS_{\ge 2}),$$
where
\begin{table}[ht] 
\begin{center}
{\openup 3pt
\begin{equation*}
\begin{array}{lll}
\toprule
q(\calR)  &=&
\prob( g \in \calR \mbox{\ and\ } \tracecycle(g) = \true)  \\
q(\calS_0)&= &
\prob(g \in \calS_0 \mbox{\ and\ } \tracecycle(g) = \true)  \\
q(\calS_1^+)&= &
\prob(g \in \calS_1^+ \mbox{\ and\ } \tracecycle(g) = \true)   \\
q(\calS_1^-)&= &
\prob(g \in \calS_1^- \mbox{\ and\ } \tracecycle(g) = \true)  \\
q(\calS_{\ge 2})&=&
\prob(g \in \calS_{\ge 2} \mbox{\ and\ } \tracecycle(g) = \true).\\
\bottomrule
\end{array}
\end{equation*}
\caption{Subdivision of the probability $q$ of (\ref{eq:q}).}
\label{tbl:subdivision*}
}
\end{center}
\end{table}
We estimate these proportions in Sections~\ref{sec:S0} - \ref{f-1}. 
Recall the definition of $\ell$ in (\ref{eq:ell}), and that $\ell>1$.
Define $b_M(\calR) = 
 \left( \frac{33}{8}\right)^M$ and note that $q(\calR)=\prob( g\in\calR)\cdot\prob(  \tracecycle(g) = \true \mid g\in\calR )\leq \prob(  \tracecycle(g) = \true \mid g\in\calR )$. Then
Proposition \ref{f2bound} gives
\begin{eqnarray*}
q(\calR) &\le&
 \frac{b_M(\calR)}{n^{M(1-s)}} \le
 \frac{b_M(\calR)}{n^{\ell}}.
\end{eqnarray*}
Define $b_M(\calS_0) = 
a_{\delta} c_{\delta}^2 r^{2s+2\delta}72$. Then 
Proposition~\ref{prop:f0} and (\ref{eq:cdelta}) give
\begin{eqnarray*}
q(\calS_0) \le \frac{b_M(\calS_0)}{n^{3-2s-2\delta}}
\le \frac{b_M(\calS_0)}{n^{\ell}}.
\end{eqnarray*}
Define $b_M(\calS_1^+) = 
a_{\delta} c_{\delta}^3 r^{3\delta}6.24.$ Then
Proposition \ref{prop:f1} and (\ref{eq:cdelta}) give
\begin{eqnarray*}
q(\calS_1^+) \le \frac{b_M(\calS_1^+)}{n^{1+s-3\delta}}
\le \frac{b_M(\calS_1^+)}{n^\ell}.
\end{eqnarray*}
Define $b_M(\calS_{\ge 2}) =  c_{\delta}^2 r^{2\delta-2s}.$ Then
Proposition \ref{prop:Sge2} gives  
\begin{eqnarray*}
q(\calS_{\ge 2}) \le \frac{b_M(\calS_{\ge 2})}{n^{2s-2\delta}}
\le \frac{b_M(\calS_{\ge 2})}{n^{\ell}}.
\end{eqnarray*}
Define $b_M(\calS_1^-) = \left(\frac{31}{r^{1-s}}\right)^M.$
Then Proposition \ref{prop:f1bound}(b) yields
\begin{eqnarray*}
q(\calS_1^-) \le \frac{b_M(\calS_1^-)}{n^{M(1-s)}}
 \le \frac{b_M(\calS_1^-)}{n^{\ell}}.
\end{eqnarray*}
Thus by (\ref{eq:mu}),
\begin{eqnarray*}
\lefteqn{b_M(\calR) +
b_M(\calS_0) + b_M(\calS_1^+) + b_M(\calS_{\ge 2}) + 
b_M(\calS_1^-)}\\ 
&\le& 
 \left( \frac{33}{8}\right)^M + 
a_{\delta} c_{\delta}^2 r^{2s+2\delta}72 + 
a_{\delta} c_{\delta}^3 r^{3\delta}6.24 + 
\frac{c_{\delta}^2}{r^{2s-2\delta}} + 
 \left(\frac{31}{r^{1-s}}\right)^M\\ 
&=&b_M\\
\end{eqnarray*}
and
\begin{equation}\label{eq:q2}
q \le \frac{b_M}{n^{\ell}}.
\end{equation}

\begin{remark}\label{rem:nineq}
{\rm
We make  a critical observation that the argument up to this point relies only on the first two inequalities of (\ref{eq:n}), and does not depend on the third inequality of (\ref{eq:n}).
} 
\end{remark}
 
By (\ref{eq:q2}) and the inequalities~(\ref{eq:probB}) and~(\ref{eq:p}), we have that 
\begin{eqnarray*}
\prob(\calB)&<& \frac{q}{1-p} \\
&\leq &\frac{b_M}{n^{\ell} \frac{\rho}{n} 
\left(\frac{n-2}{n}\right)^M }\\
&=&\frac{b_M}{\rho} \left(\frac{n}{n-2}\right)^M  \frac{1}{n^{\ell-1}}.\\ 
\end{eqnarray*}
We showed above that $\left(\tfrac{n}{n-2}\right)^M\leq\tfrac{9}{8}\leq 5\rho$. Thus $\prob(\calB)<\tfrac{5b_M}{n^{\ell-1}}$. By assumption  $n\geq \left(\frac{10b_M}{\varepsilon}\right)^{1/(\ell-1)}$ and so this is at most $\varepsilon/2$. Hence  $\prob(\calB) <  \varepsilon/2.$ 
\end{proof}

The proof of Theorem~\ref{main-theorem}(a) requires a short argument
applying Theorem~\ref{main-theorem}(b). 

\begin{proof}[Proof of Theorem~$\ref{main-theorem}(a)$]
We use the algorithm 
{\sc TraceCycle} with 
$M=4$. 
Note first that the probability that a random element 
$h\in H$ corresponds to an element $g\in G$ containing an $m$-cycle,
given that the $h$-cycles containing four
random $k$-subsets $\lambda_1, \ldots,
\lambda_4$ all have lengths of the form $r_im$ with $r_i \mid r$, is 
$\prob(g\in \calN \mid  \tracecycle(g) = \true ).$ 
Recall the definition of $q$ in (\ref{eq:q}).
Then
\begin{eqnarray*}
\lefteqn{\prob(g\in \calN \mid  \tracecycle(g) = \true )}\\
&=& \frac{\prob(g\in\calN\ \mbox{and}\ \tracecycle(g) = \true)}{\prob(\tracecycle(g) = \true)} \\
&=& \frac{\prob(\tracecycle(g) = \true) - q}{ \prob(\tracecycle(g) = \true)}\\
&=& 1 - \frac{q}{ \prob(\tracecycle(g) = \true)} \\
&\ge& 1 - \frac{q}
{ \prob(g \in \calN_{good} \mbox{\ and\ } \tracecycle(g) = \true) } \\
&=& 1 - \frac{q}
{ \prob(\tracecycle(g) = \true \mid g \in \calN_{good})\cdot \prob( g\in
  \calN_{good}) }. \\
\end{eqnarray*}

Set  $s=\frac{5}{8}$, 
 $\delta=\tfrac{1}{24}$ and 
let $\ell = 1 + \frac{1}{6}.$  Note that 
$\ell = \min\left\{ M(1-s),\right.$ $\left. 3-2s - 2\delta,
1+s-3\delta, 2s-2\delta\right\}$, so in particular the inequalities 
(\ref{eq:ds}) and (\ref{eq:ell}) all hold.   
We choose $N_0$ to be the least natural number for which
inequality~(\ref{eq:n}) holds.
 Hence the inequality (\ref{eq:M}) 
holds and in particular also $12 (rn)^s + 6 \le n$ 
and  $(rn)^s \log(n) \le n.$

Inequality~(\ref{eq:q2}) holds by Remark~\ref{rem:nineq}, so we have
$q \le
\frac{b_4}{n^{\ell}},$ where, since $M=4$, the constant $b_4$ given by
(\ref{eq:mu}), satisfies 
\begin{eqnarray*}
b_4 &\le& 
 \left( \frac{33}{8}\right)^4 + 
72\,a_{\delta} c_{\delta}^2 r^{5/3} + 
6.24\,a_{\delta} c_{\delta}^3 r^{3/8} + 
\frac{ c_{\delta}^2}{r^{7/6}} + 
 \left(\frac{31}{r^{7/24}}\right)^4.\\
\end{eqnarray*}

By Proposition~\ref{prop:mcyc} we have that
$ \prob( \tracecycle(g) = \true \mid g \in \calN_{good} ) \ge
\left(\frac{n-2}{n} \right)^4.$ 
Also, by Equation~(\ref{eq:star}), $\prob( g\in   \calN_{good})  =
\frac{\rho(G,n,m)}{m}.$  
 Hence, using 
$n \ge N_0,$ and the displayed inequality above, we have
\begin{eqnarray*}
\lefteqn{\prob(g\in \calN \mid  \tracecycle(g) = \true )}\\
&\geq& 1 - \frac{b_4}{n^{1 + \frac{1}{6}}}
  \left(\frac{n}{n-2}\right)^4 \frac{m}{\rho(G, n, m)}
\\
&\ge& 1 - \left( \frac{N_0}{N_0-2}\right)^4 \frac{
b_4}{\rho(G,n,m)} \cdot\frac{1}{n^{\frac{1}{6}}} 
= 1 - \frac{ c}{n^{\frac{1}{6}}},
\\
\end{eqnarray*}
where $c= \left( \frac{N_0}{N_0-2}\right)^4 \frac{
b_4}{\rho(G,n,m)}.$
\end{proof}

\section{Preliminaries}\label{sec:prelim}

It is useful to collect together some of the arithmetic facts we use
in the rather delicate estimations in the remaining sections. 

\begin{lemma}\label{lem:d}
Let $n,m,r$ be as in one of the lines of Table~{\rm\ref{tbl:elements}},
and let $d$ be a divisor of $rm$ with $d \le n.$ Then either $d=m$, or
$d \le 2m/7$, or $r,d$ are as in Table~$\ref{tbl:rd}$.
\begin{table}[ht]
\begin{center}
\begin{tabular}{lclll}
\toprule
$r$ & \multicolumn{4}{c}{$d$} \\
\midrule
\vspace*{.2cm}
$1$ & & $\frac{m}{3}$ & $\frac{m}{2}$& \\
\vspace*{.2cm}
$2$ & & $\frac{m}{3}$ &  $\frac{2m}{5}$ & $ \frac{2m}{3}$\\
$3$ & &$\frac{3m}{5}$ & $\frac{3m}{7}$&  \\
\bottomrule
\end{tabular}\\
\vspace*{.2cm}
\caption{possibilities for $r$ and $d$}
\label{tbl:rd}
\end{center}
\end{table}
In particular, either $d\le 2m/7$ or $d$ is one of at most $3$ different
divisors of $rm$ greater than $2m/7$ and in the latter case $d \le 2m/3 \le 2n/3.$
\end{lemma}

\begin{proof}
We have $d = r_0 \frac{m}{j}$, where $r_0$ divides $r$ and $j$ divides
$m$. If $j=1$ then $d=m$ since $2m \ge 2(n-6) > n$. So assume $j \ge
2.$ Assume also that  $d > 2m/7,$ or equivalently $7r_0 > 2j.$ If $m$
is even, then (see Table~\ref{tbl:elements}) $r=1$. Hence $r_0=1$ and
$j \le 2.$ Thus $d=m/2$ or $m/3$ as in Table~\ref{tbl:rd}. So assume
now that $m$ is odd, so $j\ge 3.$ If $j=3$ then we have the examples
$(r,d) = (1,\frac{m}{3}), (2,\frac{m}{3}), (1,\frac{2m}{3})$ in
Table~\ref{tbl:rd} and no others since if $r=3$ then (see
Table~\ref{tbl:elements}) $\gcd(m,6) = 1$.  Now assume that $j\ge 5.$
Then $r_0 > 1$ and we find $(r,d) = (2,\frac{2m}{5}),
(3,\frac{3m}{5}), (3,\frac{3m}{7})$ in Table~\ref{tbl:rd} and no
others (since $\gcd(m,6)=1$ when $r=3.$).
\end{proof}

The next result follows from the fact that $\log(1-p)> -p$ for $0<p<1$.

\begin{lemma}\label{lem:eps}
Let $\varepsilon, p$ be real numbers such that $0<\varepsilon<1$ and $0<p<1$.
Set
\[
N(\varepsilon,p):=\left\lceil \frac{\log(\varepsilon^{-1})}{p}\right\rceil.
\]
If $m\geq N(\varepsilon,p)$ then  $(1-p)^m\leq \varepsilon$.
\end{lemma}

\begin{lemma}\label{lem:simple}
Let $s$ be a real number with $\frac{1}{2} < s < 1$ and $n, r, t$
positive integers 
such that $12{(rn)}^s + 6 \le n$.
Then
\begin{itemize}
\item[(i)] $m^s/n < n^s/n < {(rn)}^s/n < 1/12$.
\item[(ii)] $n \ge 156.$
\item[(iii)] $2{(rn)}^s-t > \tfrac{24-t}{12} (rn)^s.$
\item[(iv)] if $s = 2/3$ then $n \ge 1746.$
\end{itemize}
\end{lemma}

\begin{proof} (i) This follows directly from $12 {(rm)}^s < 
12{(rn)}^s < 12{(rn)}^s + 6 \le n.$
(ii) As $s > 1/2$ and $r \ge 1$ we have 
$12 \sqrt{n} + 6 \le 12 \sqrt{rn}+6 < 12{(rn)}^s + 6 \le n.$ An easy
  calculation shows that this   implies $n \ge 156.$  
(iii) Note that $n\ge 156$ implies $n^s > n^{1/2} \ge \sqrt{156}> 12$ and
so $2{(rn)}^s -t  = (2r^s-\frac{t}{n^s}) n^s  >
(2r^s-\frac{t}{12}) n^s =\frac{24r^s-t}{12} n^s \ge\frac{24-t}{12} r^sn^s.$
(iv) By calculator.
\end{proof}

The next inequalities are easily verified.

\begin{lemma}\label{lem:ns}
Let $x\in \R$ with $x > 12.$ Then 
\begin{itemize}
\item[(a)] $x \left( \frac{1}{2}\right)^x <\frac{1}{4 x},$ and
\item[(b)]$\left( \frac{11}{12}\right)^x <\frac{5}{ x}.$
\end{itemize}
\end{lemma}

%
%

For the estimates in our last arithmetic result Lemma~\ref{lem:int1},
we first restate how to estimate sums via integrals.

\begin{lemma}\label{lem:inte}
Let $a,b\in\mathbb{Z}$ with $a<b$, and let 
$f(x)$ be a function defined on the interval
$[a-1,b+1]$, satisfying one of the lines of Table~\ref{tbl:f}. 
Then 
$$\sum_{x=a}^{b} f(x)\le \int_{a-\delta}^{b+\varepsilon} f(t) {\rm d}t.
$$
\end{lemma}

\begin{table}[ht]
\begin{center}
\begin{tabular}{p{9cm}clll}
\toprule
conditions on $f$ & & $\delta$ & $\varepsilon$ \\
\midrule
increasing in $[a,b+1]$ & & $0$ & $1$ \\
decreasing in $[a-1,b]$ & & $1$ & $0$ \\
non-negative in $[a-1,b+1]$ and for some $c\in(a,b)$
decreasing in  $[a-1,c]$ and increasing in $[c,b+1]$ & & $1$ & $1$ \\
\bottomrule
\end{tabular}
\vspace*{.3cm}
\caption{Conditions on $f$}\label{tbl:f}
\end{center}
\end{table}

%

\begin{lemma}\label{lem:int1}
Let $a,c \in \R^+$ and $n\in\Z^+$ with $n>a>c+2\ge 3$, and let $t, \ell \in \Z^+$ with
$t \ge 2$ and  $t \ge \ell.$ 
Then, summing over integers $x$ in the interval $(a,n]$,
\begin{eqnarray*}
\sum_{a < x \le n} \frac{x^{t}}{(x-c)^\ell} & < &  
\sum_{i=0}^{\ell-2} \binom{t}{i} 
\frac{c^{t-i}(a-1-c)^{i+1-\ell}}{\ell-i-1} \\
&&
\!\!\!\!\!\!\!\!\!\!\!\!\!\! + \binom{t}{\ell-1} c^{t+1-\ell} \log (n)
 + \sum_{i=\ell}^{t} \binom{t}{i} 
\frac{c^{t-i}(n+1-c)^{i+1-\ell}}{i+1-\ell}.\\\end{eqnarray*}
\end{lemma}

\begin{proof}

Note first that if $t > \ell$ the function $f(x) = \frac{x^t}{(x-c)^\ell}$ is 
decreasing on $(a, \frac{tc}{t-\ell}]$ and increasing on 
$[\frac{tc}{t-\ell},n]$, while if $t=\ell$ then $f(x)$ is decreasing
on $(a, n]$. 
In either case, by  Lemma~\ref{lem:inte} 
we have $\sum_{a < x \le n} f(x) < \int_{a-1}^{n+1} f(x) dx.$ Now
\begin{eqnarray*}
\lefteqn{\int_{a-1}^{n+1} \frac{x^t}{(x-c)^\ell} dx
 = \int_{a-1-c}^{n+1-c} \frac{(y+c)^t}{y^\ell} dy
= \int_{a-1-c}^{n+1-c} y^{-\ell} \sum_{i=0}^t\binom{t}{i}y^ic^{t-i}
 dy}\\
&=&  \sum_{i=0}^t\binom{t}{i} c^{t-i}\int_{a-1-c}^{n+1-c} y^{i-\ell}dy\\
&=&
\left[
\sum_{0\le i \le t, i \neq \ell -1} \binom{t}{i} c^{t-i}
 \frac{y^{i+1-\ell}}{i+1-\ell}
+ \binom{t}{\ell-1} c^{t+1-\ell} \log y \right]_{y=a-1-c}^{n+1-c}\\
&=& \sum_{0\le i \le t, i \neq \ell-1} \binom{t}{i} c^{t-i}
\frac{(n+1-c)^{i+1-\ell}-(a-1-c)^{i+1-\ell}}{i+1-\ell}\\
&& +  \binom{t}{\ell-1} c^{t+1-\ell} (\log (n+1-c) -\log(a-1-c))\\
&<&\sum_{i=0}^{\ell-2} \binom{t}{i} c^{t-i}
\frac{(a-1-c)^{i+1-\ell}}{\ell-i-1}
 +  \binom{t}{\ell-1} c^{t+1-\ell} \log (n) \\
& &  + \sum_{i=\ell}^{t} \binom{t}{i} c^{t-i}
\frac{(n+1-c)^{i+1-\ell}}{i+1-\ell}\\
\end{eqnarray*}

\end{proof}

\section{Binomial inequalities and partitions}
\label{sec:binomial}

In this section we prove a result about partitions that will be needed 
in Sections~\ref{f-1} and \ref{sec:tracing}. As preparation, we prove an
inequality about certain binomial coefficients.

\begin{lemma}\label{lem:binom}
Let $a$ be an integer such that $a>1$, and let $c,\ell$ be integers such 
that $1\leq \ell< c$. Then
\[
\binom{ca-1}{a-1}\binom{c}{\ell}\leq \binom{ca}{\ell a}.
\]
\end{lemma}

\begin{proof}
The proof is by induction on $\ell$, for fixed $c, a$. Since 
$\binom{c}{\ell}=
\binom{c}{c-\ell}$ and $\binom{ca}{\ell a}=\binom{ca}{(c-\ell)a}$, it is
sufficient to prove this for $1\leq \ell\leq \lfloor c/2\rfloor$. Suppose
first that $\ell=1$. Here it is straightforward to check that
\[
\binom{ca-1}{a-1}\binom{c}{1}= \binom{ca}{a}.
\]
Now suppose that $1\leq\ell< \lfloor c/2\rfloor$ and that the inequality
holds for $\ell$. Then, using induction we have
\[
\binom{ca-1}{a-1}\binom{c}{\ell +1}=\binom{ca-1}{a-1}\binom{c}{\ell}\frac
{c-\ell}{\ell+1}\leq \binom{ca}{\ell a}\frac{c-\ell}{\ell+1}.
\]
This latter quantity is at most $\binom{ca}{(\ell+1)a}$ if and only if
\begin{equation}\label{bin1}
\frac{c-\ell}{\ell+1}\cdot\frac{1}{(\ell a)!(ca-\ell 
a)!}\leq\frac{1}{(\ell a+a)!(ca-\ell a-a)!}
\end{equation}
and this is equivalent to
\[
\frac{c-\ell}{\ell+1}\leq \frac{ca-\ell a}{\ell a+a}\cdot \frac{ca-\ell 
a-1}{
\ell a+a-1}\dots\frac{ca-\ell a -a+1}{\ell a+1}.
\]
Now the first factor on the right hand side is equal to 
$(c-\ell)/(\ell+1)$,
and each of the other factors is at least 1 since $c\geq 2\ell+1$. Thus the
inequality (\ref{bin1}) holds, and so the induction proof is complete.
\end{proof}

\begin{lemma}\label{lem:Z}
\begin{enumerate}
\item[(a)] For  $2 \le k \le d < n$ we have
\[\binom{d}{k}\le\left(\frac{d}{n}\right)^k \binom{n}{k}\]
and moreover, if $d\le \alpha n$ for some $\alpha < 1$ then
\[\frac{\binom{d}{k}}{\binom{n}{k}}\le \alpha^{k-1} 
\frac{d-k+1}{n-k+1}\le \alpha^k.\] 
\item[(b)] For $2\le k \le 2n/3$ we have
$$\binom{\lfloor n/2\rfloor}{\lfloor k/2\rfloor} < 2\binom{n}{k} 
\left( \frac{3k}{4n}\right)^{\lceil k/2\rceil}.$$
\end{enumerate}
\end{lemma}
\begin{proof}
Every part of the proof depends on the following observation:

\noindent Fact 1: For $0\le i \le t \le n$ with $i < n$ we have
$\frac{t-i}{n-i} \le \frac{t}{n}$ with strict inequality if $t < n.$

For $(a)$ observe that
$$\frac{\binom{v}{k}}{\binom{n}{k}}= \prod_{i=0}^{k-1}\frac{v-i}{n-i}
\le \prod_{i=0}^{k-1}\frac{v}{n} = \left(\frac{v}{n}\right)^k.$$

If $d \le \alpha n$ for some $\alpha < 1$ then
\begin{eqnarray*}
\frac{\binom{d}{k} }{\binom{n}{k}} &=& 
\left(\prod_{i=0}^{k-2} \frac{d-i}{n-i}\right) \cdot \frac{d-k+1}{n-k+1}
\le 
\left( \prod_{i=0}^{k-2} \frac{\alpha n-i}{n-i}\right) \cdot \frac{d-k+1}{n-k+1}.\\
\end{eqnarray*}
Now, again by Fact~1,
 $(\alpha n -i)/(n-i) \le \alpha$ 
and $\frac{d-k+1}{n-k+1} \le \frac{d}{n} \le \alpha,$
and therefore 
$$\frac{\binom{d}{k} }{\binom{n}{k}} 
\le {\alpha^{k-1}}  \frac{d-k+1}{n-k-1} \le \alpha^k.
$$

For $(b)$ let $n_0 = \lfloor n/2\rfloor$ and $k_0 = \lfloor k/2 \rfloor$. 
Note then that
\begin{eqnarray*}
\frac{\binom{n_0}{k_0}}{\binom{n}{k}} &=& 
\frac{n_0(n_0-1) \cdots (n_0-k_0+1)}{n(n-1)\cdots (n-k+1)} k(k-1) \cdots (k_0+1)\\
&=&  \prod_{i=0}^{k_0-1}\frac{n_0-i}{n-i} \cdot \prod_{j=k_0}^{k-1} \frac{k+k_0-j}{n-j}.
\end{eqnarray*}
Now $k + k_0 \le 2n/3 + n/3 \le n.$ 
Applying Fact~1 with $t = n_0$ to the first product and with
$t = k+k_0$ to the second, we obtain
\begin{eqnarray*}
\frac{\binom{n_0}{k_0}}{\binom{n}{k}} &\le& 
 \prod_{i=0}^{k_0-1}\frac{n_0}{n} \cdot \prod_{j=k_0}^{k-1} \frac{k+k_0}{n}\\
&=&  \left(\frac{n_0}{n}\right)^{k_0} \cdot \left( \frac{k+k_0}{n} \right)^{k-k_0}\\
&\le&  \left(\frac{1}{2}\right)^{k_0} \cdot \left( \frac{3k}{2n} \right)^{k-k_0}\\
&=&  \frac{3^{\lceil k/2\rceil}}{2^k} \cdot \left(\frac{k}{n} \right)^{\lceil k/2\rceil}\\
&\le &  \frac{2\cdot 3^{\lceil k/2\rceil}}{4^{\lceil k/2\rceil}} \cdot \left(\frac{k}{n} \right)^{\lceil k/2\rceil}.\\
\end{eqnarray*}
Note that the first inequality is strict if either $k_0 \ge 2$ or $k-1
> k_0$, that is, if $k \ge 3.$ If $k = 2$ then $\binom{n_0}{k_0} =
\lfloor n/2\rfloor,$ while 
$2\binom{n}{k}\left( \frac{3k}{4n} \right)^{\lceil k/2\rceil} =
\frac{3}{2}(n-1) > \lfloor n/2\rfloor.$  Thus (b) is proved for all $k$.

\end{proof}

\begin{lemma}\label{lem:ZZ}
Let $d, k, t$ be positive integers and $a >0$ such that $k \le d$ and
$\frac{t}{d-k+1} \le a.$ Then
\begin{eqnarray*}
\lefteqn{
(d+t)(d+t-1)\dots (d+t-k+1) }\\
&<&d(d-1) \dots (d-k+1) (1+\frac{(1+a)^kt}{a(d-k+1)}).
\end{eqnarray*}
\end{lemma}

\begin{proof}
Note first that
\begin{eqnarray*}
\lefteqn{(d+t)(d+t-1)\dots (d+t-k+1)}\\
 &=  &
d(1+\frac{t}{d})(d-1)(1+\frac{t}{(d-1)})\dots (d-k+1) (1+\frac{t}{(d-k+1)}) \\
& \le & d (d-1) \dots (d-k+1) ( 1 + \frac{t}{(d-k+1)})^k.  \\
\end{eqnarray*}
Set $x = \frac{t}{d-k+1},$ so $0 < x \le a$. Then 
\begin{eqnarray*}
(1+x)^k&=&\sum_{j=0}^k \binom{k}{j} x^j = 
1+x\sum_{j=1}^k\binom{k}{j}x^{j-1}\\
&\le&1+x\sum_{j=1}^k\binom{k}{j}a^{j-1}\\
&\le&1+\frac{x}{a}\sum_{j=0}^k\binom{k}{j}a^{j}\\
&=&1+\frac{x}{a}(1 + a)^k.
\end{eqnarray*}

\end{proof}

Now we state and prove the result on partitions.
\begin{proposition}\label{npk}
Let $\calU$ be a finite set of size $u>1$, and let $\calP$ be a 
partition of
$\calU$ in which all parts have size at least $2$. For $2 \le k_0\le u$, let
$N_\calP(k_0)$ denote the number of $k_0$-subsets of $\calU$ that are 
unions of
parts of $\calP$. Then
$N_{\calP}(k_0)\leq \binom{\lfloor u/2\rfloor}{\lfloor k_0/2
\rfloor}$, and moreover, if $k_0$ is odd and $u$ is even, then $u\geq 4$ and
$N_{\calP}(k_0)\leq\binom{(u-2)/2}{(k_0-1)/2}$.
In particular, $N_{\calP}(k_0) = 1$
if $k_0 = u$ and
$N_{\calP}(k_0) \leq \frac{1}{u-1}\binom{u}{k_0}$ otherwise.
\end{proposition}

\begin{proof}
First we construct a partition $\calP'$ of $\calU$ having at most two parts
of size 1, and all parts of size at most 2.
Start with $\calP'=\emptyset$ and run through the parts of $\calP$.
For each part $P\in\calP$ of even size, choose any partition of
$P$ with all parts of size 2, and add the parts of this partition to 
$\calP'$.
If all parts of $\calP$ have even size, then the construction of
$\calP'$ is completed in this way.
So suppose that $\calP$ has at least one part of odd size.  In this case
$\calP'$ will have 1 or 2 parts of size 1, and its construction is 
completed as follows.  For each part
$P\in\calP$ of odd size $p:=|P|$, add $(p-1)/2$
parts of size 2 to $\calP'$  formed from $p-1$ of the points of $P$.
Let $P_1,\dots,P_r$ be the odd length parts of $\calP$. Pair up the
remaining $r$ points into
parts of size 2 and add them to $\calP'$,
leaving exactly 1 or 2 of these points to form singleton parts of
$\calP'$.

Next  we define, for  each $k_0$-subset  $\eta$ of  $\calU$ that  is a
union of parts  of $\calP$, a $k_0$-subset $\eta'$ that  is a union of
parts of $\calP'$. Note that if  $k_0$ is odd then $\eta$ must contain
a part of  $\calP$ of odd size,  and in this case $\calP'$  has one or
two singleton  parts. If $k_0$ is  odd and $\calP'$  has two singleton
parts, then  we choose one  of them, and  we always place  this chosen
singleton part in  $\eta'$.  To define $\eta'$ for  a given $\eta$, we
start with  $\eta'=\emptyset$ and build  it up by considering  in turn
each of  the parts $P$  of $\calP$ contained  in $\eta$.  If  $|P|$ is
even, then $P$ is  a union of parts of $\calP'$ of  size 2, and we add
all  of these  parts to  $\eta'$.  If  $|P|$ is  odd, then  we  add to
$\eta'$ all the parts of size  2 of $\calP'$ contained in $P$. At this
stage  $|\eta'|=k_0-\ell$, where  $\ell$ is  the number  of  odd sized
parts of  $\calP$ contained in $\eta$.   Next we add to  $\eta'$ up to
$\lfloor\ell/2\rfloor$ parts of $\calP'$ of size 2 that contain points
from two different parts of $\calP$. If $\eta'$ cannot be completed in
this way then either (i) $\ell$ is  odd, or (ii) $\ell$ is even and is
equal to the number of odd  sized parts of $\calP$. Case (i) occurs if
and only  if $k_0$ is odd, and  here we add to  $\eta'$ the designated
singleton part of $\calP'$. In case (ii) there are two singleton parts
of $\calP'$, and we add to $\eta'$ these two singleton parts.

Note  that, if  $\ell\geq2$,  then we  may  have had  some freedom  in
choosing the  $\lfloor\ell/2\rfloor$ parts of $\calP'$ of  size 2 that
contain points from two different parts of $\calP$, so $\eta'$ may not
be determined  uniquely by $\eta$.  On the other hand,  $\eta'$ always
determines $\eta$ uniquely, since $\eta$  is the union of the parts of
$\calP$ that have  at least two points in  $\eta'$. Thus distinct sets
$\eta$ correspond to distinct sets $\eta'$.

It follows that $N_{\calP}(k_0)\leq N'$ where $N'$ is the
number of $k_0$-subsets $\ga\subseteq\calU$ such that $\ga$ is a union
of parts of $\calP'$ and in addition, if $k_0$ is odd and $\calP'$ has
two singleton parts, then $\ga$ contains a designated one of these
singleton parts.

Suppose that  $\ga$ is such a  $k_0$-subset.  If $\calP'$  has at most
one part of size 1,  then $\ga$ contains $\lfloor k_0/2\rfloor$ of the
parts of  $\calP'$ of size  2 (and also  a singleton part if  $k_0$ is
odd).        Thus       $N'\leq\binom{\lfloor      u/2\rfloor}{\lfloor
k_0/2\rfloor}$.  Note that  in  this  case, if  $k_0$  were odd,  then
$\calP$ would have  at least one odd part, and  so $\calP'$ would have
exactly one  odd part,  whence $u$  would be odd.  Thus the first assertion is
proved  in this  case.  So  suppose  that $\calP'$  has two  singleton
parts, in which case $u$ is even.  If $k_0$ is odd then $k_0 \ge 3$ and
$\ga$ consists
of $\lfloor k_0/2\rfloor$  of the parts of $\calP'$ of  size 2 and the
designated      singleton      part,      whence     $u\geq4$      and
$N'\leq\binom{(u-2)/2}{(k_0-1)/2}<  \binom{\lfloor u/2\rfloor}{\lfloor
k_0/2\rfloor}$.   On the  other  hand,  if $k_0$  is  even then  $\ga$
consists of  $k_0/2$ of  the two-point parts  (or $k_0/2 -1$  parts of
size  two and  the two  singleton parts).  Again $N'\leq\binom{\lfloor
u/2\rfloor}{\lfloor k_0/2\rfloor}$. This proves the first assertion in
all cases.

Note that   $\lfloor u/2\rfloor = \lfloor k_0/2\rfloor$ if and only if
either $k_0 = u$, or $k_0 = u-1$ with $u$ odd. If $k_0 = u$ 
obviously $N_{\calP}(k_0)= N' = 1.$ If $k_0 = u-1$ with $u$ odd
then ${\calP'}$ has a unique part of size 1 and its complement is the
unique $k_0$-subset of $U$ that is a union of parts of ${\calP'}$ -
it may or may not be a union of parts of ${\calP}.$ Thus 
$N_{\calP}(k_0)\leq N' = 1 \le \frac{1}{u-1} \binom{u}{k_0}.$

So suppose from now on that 
 $\lfloor k_0/2\rfloor < \lfloor u/2\rfloor,$ and 
set $u_1=\lfloor u/2\rfloor$ and $k_1=\lfloor k_0/2\rfloor$. Then
$ \binom{\lfloor u/2\rfloor}{\lfloor k_0/2\rfloor}=\binom{u_1}{k_1}$, 
and by
Lemma~\ref{lem:binom}, this is at most $\frac{1}{2u_1-1}\binom{2u_1}{2k_1}$.
If $k_0$ and $u$ are  even, then $k_0 < u$ and this quantity is
at most $\frac{1}{u-1}\binom{u}{k_0}$. 
If $k_0$ is even and $u$ is odd, then $2 \le k_0$ and
$\frac{1}{2u_1-1}\binom{2u_1}{2k_1} = \frac{1}{u-2}\binom{u-1}{k_0}.$
This in turn is at most  $\frac{1}{u-1}\binom{u}{k_0}$. 
Now suppose $k_0$ and $u$ are odd. Then $\lfloor k_0/2\rfloor  <
 \lfloor u/2 \rfloor$ implies $k_0 \le
u-2$, and
$\frac{1}{2u_1-1}\binom{2u_1}{2k_1} =
\frac{1}{u-2}\binom{u-1}{k_0-1}$ which
is at most $\frac{1}{u-1}\binom{u}{k_0}$. 
Finally consider $k_0$ 
odd and $u$ is even. As shown above $u\geq  4$ and
$N' \le \binom{(u-2)/2} {(k_0-1)/2}$. By Lemma~\ref{lem:binom}, this is at most
  $\frac{1}{u-3}\binom{u-2}{k_0-1}$, which in turn is at most
$\frac{1}{u-1}\binom{u}{k_0}$.
\end{proof}

For a prime $p$ and an integer $n$, let $n_p$ denote the $p$-part of
$n$, that is the highest power of $p$ dividing $n$.
Recall that, for a positive integer $k_0\leq n$,
a $k_0$-subset $\ga'$ of $\Om$, and an element $g\in S_n$, we denote by 
$c_{k_0}(\ga',g)$ the length of the $g$-cycle containing $\ga'$ in the 
action of $g$ on $k_0$-sets.

\begin{lemma}\label{pc-1}
Let $g\in S_n$, let $C$ be a $g$-cycle of length $t$,  let 
$k_0$ be a positive integer such that $k_0\leq t$ and let $p$ be a prime dividing $t$.
\begin{itemize}
 \item[(a)] Suppose that $\ga'$ is a  $k_0$-subset of $C$ such 
that the $p$-part $t_p$ does not divide $c_{k_0}(\ga',g)$. Then 
$\ga'$ is a union of $Z(C,p)$-orbits, where $Z(C,p)$ is the 
subgroup  of order $p$ of the cyclic group $\la g^C\ra\cong Z_t$ 
induced by $g$ on $C$. In particular $p$ 
divides $\gcd(k_0,t)$.

\item[(b)] The number $\sigma(k_0,C)$ of $k_0$-subsets $\ga'$ of $C$ 
such that $t_p$ does not divide $c_{k_0}(\ga',g)$ is
at most $\binom{\lfloor t/2\rfloor}{\lfloor k_0/2
\rfloor}$, and in particular, is 
 $1$ if $k_0=t$, and at most 
$\frac{1}{t-1}\binom{t}{k_0}$ if $k_0<t$.

\end{itemize}
\end{lemma}

\begin{proof}
(a) Since  $t_p$ does not divide $c_{k_0}(\ga',g)$ and  $\la g^C\ra\cong Z_t$,
it follows that the setwise stabiliser $H$ of $\ga'$ in  $\la g^C\ra$
contains the unique subgroup $Z(C,p)$ of $\la g^C\ra$ of order $p$. As $\ga'$ 
is $H$-invariant, $\ga'$ is a union of $H$-orbits in $C$, and hence 
$\ga'$ is a union of $Z(C,p)$-orbits in $C$. In particular, $p$ divides $k_0$ as well as $t$.

(b) If $k_0=t$ then $C$ is its unique $k_0$-subset and  $\sigma(k_0,C)=1$.
If $k_0<t$ then, by Proposition~\ref{npk},  $\sigma(k_0,C)\leq \binom{\lfloor t/2\rfloor}{\lfloor k_0/2
\rfloor}$ and also $\sigma(k_0,C)\leq
\frac{1}{t-1}\binom{t}{k_0}$.
\end{proof}

\begin{corollary}\label{pc}
Let $G,n,m,r$ be as in one of the lines of Table~{\rm\ref{tbl:elements}},
and let $g\in G$. Let $\Si(g)$ be as in Table~{\rm\ref{tbl:notation}}
with $u=|\Si(g)|$, and let $k_0$ be a positive integer such that $k_0\leq u$. 
Then the number $\sigma(k_0,\Si(g))$
 of $k_0$-subsets $\ga'$ of $\Sigma(g)$ such that $c_{k_0}(\ga',g)$ 
divides $rm$ satisfies  
$$\sigma (k_0,\Si(g)) \le 
\begin{cases}
0 & \mbox{\ if\ } k_0=1\\
1 & \mbox{\ if\ } k_0 = u\\
\frac{1}{u-1} \binom{u}{k_0} & \mbox{\ if\ } 1< k_0 < u.\\
\end{cases}$$
\end{corollary}

\begin{proof}
For each $g$-cycle $C$ in $\Sigma(g)$, by the definition of $\Si(g)$,
 $|C|$ does not divide $rm$, and hence there exists 
a prime $p(C)$ such that $|C|_{p(C)}$ does not divide $rm$. Let 
$Z(C,p(C))$ denote the subgroup of order $p(C)$ of the cyclic group 
$\la g^C\ra$ induced by $g$ on $C$, let $\mathcal{P}(C)$ denote the
set of $Z(C,p(C))$-orbits in $C$ (all of length $p(C)$), and let $\mathcal{P}=
\cup_C\mathcal{P}(C)$ denote the corresponding partition of $\Si(g)$.

Suppose that  $\ga'$ is a $k_0$-subset of $\Sigma(g)$, and for each
$g$-cycle $C$ in $\Sigma(g)$, let $k(C)=|\ga'\cap C|$. 
Then $c_{k_0}(\ga',g)$ is the least common multiple of 
 $c_{k(C)}(\ga'\cap C,g)$, over all $g$-cycles $C$ such that $k(C)\ne0$.
Note that  $c_{k(C)}(\ga'\cap C,g)$ divides $|C|$.

Suppose now that $c_{k_0}(\ga',g)$ divides $rm$. Then for each $C$ 
such that $k(C)\ne0$, also $c_{k(C)}(\ga'\cap C,g)$ divides $rm$,
and hence $|C|_{p(C)}$ does not divide $c_{k(C)}(\ga'\cap C,g)$. 
By Lemma~\ref{pc-1}, $\ga'\cap C$ is a union of parts of $\mathcal{P}(C)$.
Thus $\ga'$ is a union of parts of $\mathcal{P}$. 
Since  all parts of $\calP$ have size at least 2, this implies that
$\sigma (k_0,\Si(g))=0$ if $k_0=1$, and 
the inequality for $1 < k_0\le u$ follows from Proposition~\ref{npk}.
\end{proof}

\section{Tracing $k$-subsets}
\label{sec:tracing}

For the remainder of this paper we assume that $k$ is an
integer with $2 \le k \le n/2.$ We use $\Delta(g)$, $\Sigma(g)$ and other notation 
introduced in Tables~\ref{tbl:notation} and~\ref{tbl:subsets}.
Further, we use without further reference the number  $M$ of independent uniformly distributed random $k$-subsets in Algorithm~$\ref{algo:TraceCycle}$ {\sc TraceCycle}, where $M$ satisfies (\ref{eq:M}), in particular $M \ge 4$. 

\begin{proposition}\label{TraceOne}
Let $G,n,m,r$ be as in one of the lines of Table~{\rm\ref{tbl:elements}},
and suppose that $g\in G$ does not contain an $m$-cycle. Set  $v=|\Del(g)|$ and 
 suppose that $v \le n - k - 1$. Then the proportion of
$k$-subsets $\ga$ of $\Om$ such that $c_k(\ga,g)=r_0m$, for some 
$r_0$ dividing $r$, is at most $\frac{v^k}{n^k}+ \frac{1}{n-v-1}.$ 
\end{proposition}

\begin{proof}
Set $u=n-v=|\Si(g)|$. Suppose that $\ga$
is a $k$-subset of $\Om$ such that $c_k(\ga,g)=r_0m$ for some 
$r_0$ dividing $r$, and set $k_0:=|\ga\cap 
\Sigma(g)|$. Then $k_0 \le \min\{ k,u\}$. 
By assumption,  $v \le n-k-1$ and so $u = n-v \ge k+1$ and  $k_0\leq \min\{k,u\} = k$. 
Also $c_{k_0}(\ga\cap\Si(g),g)$
divides $c_k(\ga,g)$, and hence divides $rm$.
By Corollary~\ref{pc}, the number $\sigma(k_0,\Si(g))$ of 
$k_0$-subsets $\ga'$ of $\Sigma(g)$ such that 
$c_{k_0}(\ga',g)$ divides $rm$ is 0 if
$k_0=1$, 1  if $k_0 = u$, and at most 
$\frac{1}{u-1} \binom{u}{k_0}$, otherwise.
If $k_0=0$ then $\gamma$ is one of the $\binom{v}{k}$ $k$-subsets of $\Delta(g).$
Thus the number of possibilities for
$\ga$ is at most
\begin{eqnarray*}
&&\binom{v}{k}+\sum_{k_0=2}^{k} \sigma(k_0,\Si(g))\binom{n-u}{k-k_0}\\
&\leq&  \binom{v}{k}+\frac{1}{u-1}\sum_{k_0=2}^{k}
\binom{u}{k_0}\binom{n-u}{k-k_0}\\ 
&<&\binom{v}{k} + \frac{1}{u-1}\binom{n}{k}.
\end{eqnarray*}
Now $u-1=n-v-1$, hence the above is  $\binom{v}{k} + \frac{1}{n-v-1}
\binom{n}{k}$. By Lemma~\ref{lem:Z}(a),  $\binom{v}{k}$ is at most
$(v/n)^k \binom{n}{k}$, which completes the proof.
\end{proof}

\begin{lemma}\label{TraceTwo}
Let $G,n,m,r$ be as in one of the lines of Table~{\rm\ref{tbl:elements}}.
Let $g$ be a uniformly distributed random element of $G$, and suppose that
$g$ does not contain an $m$-cycle, and that $v=|\Del(g)| \le n - k - 1$. 
Then the following both hold.
\begin{itemize}
\item[(a)] $\prob( \mbox{\sc TraceCycle}(g) = \true
) \le  2^M\left(\left(\frac{v^k}{n^k}\right)^M+
  \left(\frac{1}{n-v-1}\right)^M\right),$  
\item[(b)]
$\prob( \mbox{\sc TraceCycle}(g) = \true) \le  16\max\left\{\left(\frac{v}{n}\right)^4,
  \left(\frac{1}{n-v-1}\right)^4\right\}.$ 
\end{itemize}
Moreover, if $3 \le v\le n-3$ then 
$\prob( \mbox{\sc TraceCycle}(g) = \true
) \le  16\left(\frac{v}{n}\right)^4.$
\end{lemma}

\begin{proof}
Now {\sc  TraceCycle}$(g)=$ \true\  if and only  if $c_k(\ga,g)
=r_0m$, for some $r_0$ dividing $r$,
for each of the  $M$ independent  uniformly distributed random  $k$-sets $\ga$
tested during  the algorithm. Thus if  $g$ does not contain
an    $m$-cycle,    the   probability  that   {\sc  TraceCycle}$(g)=$ \true\ is
$p^M$, where $p$ is  the proportion of 
$k$-subsets $\ga$ such that $c_k(\ga,g)=r_0m$  for some $r_0$ dividing $r$. 
By Proposition~\ref{TraceOne},
$ p \le \frac{v^k}{n^k}+ \frac{1}{n-v-1}$.
Note that $p^M \le p^4$ since $p \le 1$ and
$M \ge 4$.
Set $x= \frac{v^k}{n^k}$ and $y= \frac{1}{n-v-1}$. If $x\leq y$ then $(x+y)^M\leq (2y)^M=2^My^M$, and similarly if $x\geq y$ then $(x+y)^M\leq 2^Mx^M$. It follows that $p^M\leq 2^M(x^M+y^M)$, proving part (a). 

For (b), we observe that
\[
p^M\leq p^4\leq (x+y)^4\leq ( 2\max\{x,y\})^4 = 16 \cdot
\max\{x,y\}^4. 
\]
Part (b) follows on noting that $x\leq v/n$ (since $v\leq n$). Finally suppose that 
 $3 \le v \le n-3$. Then $n \ge v+3 \ge v + 2 + \frac{2} {v-1}$ so
$n(v-1) \ge v^2 + v$ and hence $(n-v-1)v \ge n$, that is,
$\frac{v}{n} \ge \frac{1}{(n-v-1)}$. The last assertion now follows from part (b).
\end{proof}

Now we analyse  the effect of {\sc TraceCycle}  applied to elements of 
$\calR$.

\begin{proposition}\label{f2bound}
Let $G,n,m,r$ be as in one of the lines of Table~{\rm\ref{tbl:elements}}
and suppose that $12{(rn)}^s +6 \le n.$ 
Then, for a uniformly distributed random
element $g\in G$,
$$
\prob(\mbox{\sc TraceCycle}(g)\ =\mbox{\true} \mid 
g\in\calR)
 \le 
 \left( \frac{33}{8n^{1-s}}\right)^M.
$$
\end{proposition}

\begin{proof}
By definition, for $g\in\calR$, $v=|\Del(g)|\leq 4(rn)^s$
and $g$ does not contain an $m$-cycle.
 By our assumptions on $n$ and $k$ and the hypothesis, we have $n-k-1 \ge
n/2 -1 > 4(rn)^s \ge v.$

Thus by Proposition~\ref{TraceOne}, the
proportion of $k$-subsets $\gamma$ such that $c_k(\gamma, g) = r_0m$, 
for some $r_0$ dividing $r$, is
at most $\frac{v^k}{n^k} + \frac{1}{n-v-1} \le \frac{{(4r^s)}^k}{n^{k(1-s)}}
+ \frac{1}{n - 4(rn)^s -1}$.

Now {\sc TraceCycle}$(g)=$ \true\ if and only if $c_k(\ga,g)=r_0m$,
for some $r_0$ dividing $r$,  for  each of  $M$
independent uniformly distributed random $k$-sets $\ga$ tested during the
algorithm. Thus, given $g\in\calR$, the probability of this occurring is
at most
\[
\left(\frac{{(4r^s)}^k}{n^{k(1-s)}}+\frac{1}{n-4(rn)^s-1}\right)^M.
\]
Now $12(rn)^s < n$, that is to say, $\frac{4r^s}{n^{1-s}}<\frac{1}{3}$.
Also  $k\ge2$, $r \le 3$ and $s<1$.
Therefore $(\frac{4r^s}{n^{1-s}})^k \le (\frac{4r^s}{n^{1-s}})^2 
<\frac{4r^s}{{3n^{1-s}}} < \frac{4}{n^{1-s}}.$
Also,  by assumption,
$n-4(rn)^s-1 \ge 8(rn)^s +5  >  8r^sn^s > 8r^sn^{1-s}$. 
Therefore, the probability that 
{\sc TraceCycle}$(g)=$ \true\  is at most 
\[
\left(\frac{4}{n^{1-s}}+\frac{1}{8r^sn^{1-s}}\right)^M 
\leq \left( \frac{33}{8n^{1-s}}\right)^M.
\]
\end{proof}

Next we analyse  the effect of {\sc TraceCycle}  applied to elements of 
$\calN_{good}$ (defined in Table~\ref{tbl:subsets}).

\begin{lemma}\label{lem:mcyc}
Let $G,n,m,r$ be as in one of the lines of Table~{\rm\ref{tbl:elements}}, and 
let $k_0$ be an integer satisfying $0\le k_0 \le k$. Let $g\in \calN$ and let $C$ be the
$m$-cycle contained in $g$. Then the number of
$k_0$-subsets of $C$ that can occur as
$\ga\cap C$, for a $k$-subset $\ga$ of $\Om$ such that $c_k(\ga,g)$ is not divisible by 
$m$, is at most
\begin{eqnarray*}
 \sigma_{k_0} &=& \left\{\begin{array}{ll}
                    1 & \textrm{if $k_0=0$ and $k\leq n-m$}\\
                   0 & \textrm{if }\ \gcd(m,k_0)=1\ \mbox{or if}\\
		  & k_0 < k-n+m\\
		 \omega(\gcd(m,k_0))\binom{\lfloor m/2\rfloor}{\lfloor k_0/2\rfloor} & \textrm{if } \gcd(m,k_0)>1\ \mbox{and}\\
		 &k_0\geq \max\{1,k-n+m\}  
\end{array}\right.\\
\end{eqnarray*}
where $\omega(d)$ is the number of distinct prime divisors of an integer 
$d$. 
\end{lemma}

\begin{proof}
Let $\sigma'$ be the number of
$k_0$-subsets of $C$ that can occur as
$\ga\cap C$, for a $k$-subset $\ga$ of $\Om$ such that $c_k(\ga,g)$ is not divisible by 
$m$. Note that, if $\gamma$ is such a $k$-subset, then $\gamma\setminus C$ is contained in the complement $\overline C$ of $C$ and hence $k=|\gamma|\leq k_0 + |\overline C|=k_0+n-m$. 
Thus if $k_0< k-n+m$ then $\sigma'=0$. Also if $k_0=0\geq k-n+m$, then $\gamma\cap C=\emptyset$ so $\sigma'\leq 1$. Suppose now that $k_0>0$ and $k_0\geq k- n+m$,
that is, $k_0\geq \max\{1,k-n+m\}$.

Let $\gamma$ be such that $c_k(\ga,g)$ is not divisible by $m$. 
Then $c_{k_0}(\ga\cap C,g)$ properly divides $m$, and hence
there exists a prime $p$ dividing $m$
such that the $p$-part $m_p$ does not divide $c_{k_0}(\ga\cap C,g)$.
By Lemma~\ref{pc-1}(a), $p$ divides $\gcd(m,k_0)$ (and
in particular if $\gcd(m,k_0)=1$ then $\sigma'=0$). 
If such a prime $p$ exists then, by
Lemma~\ref{pc-1}(b), the number of $k_0$-subsets $\ga\cap C$ such that 
 $m_p$ does not divide $c_{k_0}(\ga\cap C,g)$ is  at most $\binom{\lfloor
  m/2\rfloor}{ 
\lfloor k_0/2\rfloor}$. Finally there are at most $\omega(\gcd(m,k_0))$ 
primes $p$ to consider, and the proof is complete.
\end{proof}

\begin{proposition}\label{prop:mcyc}
Let $G,n,m,r$ be as in one of the lines of Table~{\rm\ref{tbl:elements}}
and suppose that $g\in \calN_{good}$, and $12{(rn)}^s + 6\le n$.
Then the proportion of
$k$-subsets $\ga$ of $\Om$ such that $c_k(\ga,g)\ne mr_0$, for any $r_0$ dividing
$r$, is at most  

$$\sum_{k_0=\max\{k-(n-m),0\}}^k \sigma_{k_0} \binom{n-m}{k-k_0}/\binom{n}{k}\le \sqrt{8k} \left(\frac{3k}{4m}\right)^{\lceil k/2\rceil},
$$
where $\sigma_{k_0}$ is as in Lemma~{\rm\ref{lem:mcyc}}. Moreover, for a
uniformly distributed random element $g\in G$, 
$$\prob(  {\tracecycle}(g)= \true \mid g\in {\calN_{good}})
\ge  \left( \frac{n-2}{n}\right)^M.$$
\end{proposition}

\begin{proof}
Let $C$ denote the $m$-cycle in $g$ and let $\ga$ be a 
$k$-subset  of $\Om$ such that $c_k(\ga,g)\ne mr_0$ for any $r_0$ dividing
$r$. By the definition of $\calN_{good}$, this implies that
$c_{k_0}(\ga\cap C,g)$ is not divisible by $m$, where  $k_0=|\ga\cap C|$. Now $0\leq k_0\leq\min\{k,m\}=k$, and moreover $k_0\geq k-(n-m)$ since
$\ga\subseteq (\ga\cap C)\cup(\Omega\setminus C)$. Given $\ga\cap C$, 
there are at most $\binom{n-m}{k-k_0}$ choices for
$\ga\setminus C$. Hence, by Lemma \ref{lem:mcyc}, the number  
of such $k$-subsets $\ga$ is at most 
\begin{equation}\label{eq:X}
X:=\sum_{k_0=\max\{k-n+m,0\}}^k \sigma_{k_0}
\binom{n-m}{k-k_0} 
\end{equation}
where $\sigma_0=1$,  and
$\sigma_{k_0}=\omega((\gcd(m,k_0))\binom{\lfloor m/2\rfloor}{\lfloor
  k_0/2\rfloor}$ for $k_0>0$. Now $\omega(\gcd(m,k_0))\le 
\omega(k_0)\le \sqrt{2 k_0}\le \sqrt{2k}$ (see for example, 
\cite[p. 395]{NivenZuckermanetal91}). Hence,
$X \le \sqrt{2} k \sum_{k_0=\max\{k-n+m,0\}}^k \binom{\lfloor
  m/2\rfloor}{\lfloor k_0/2\rfloor} \binom{n-m}{k-k_0}$ 
and by Lemma \ref{lem:Z}(b), we have,

\begin{eqnarray*}
X &\le& \sqrt{2k} \sum_{k_0=\max\{k-n+m,0\}}^k 2\binom{m}{k_0}\left(\frac{3k_0}{4m}\right)^{\lceil k_0/2\rceil} \binom{n-m}{k-k_0}\\
&\le& \sqrt{8k} \left(\frac{3k}{4m}\right)^{\lceil k/2\rceil}\sum_{k_0=\max\{k-n+m,0\}}^k \binom{m}{k_0}
\binom{n-m}{k-k_0}\\
&\le&\sqrt{8k} \left(\frac{3k}{4m}\right)^{\lceil k/2\rceil}\binom{n}{k}
\end{eqnarray*}
and hence the proportion $X/\binom{n}{k}\le p$ where $p:=\sqrt{8k} \left(\frac{3k}{4m}\right)^{\lceil k/2\rceil}$.

Now we consider the final assertion. Note that
{\sc TraceCycle}$(g)=$ \true\ if and only if, for 
each of the $M$ independent uniformly distributed random $k$-subsets $\ga$
tested, we have $c_k(\ga,g)= r_0m$ for some $r_0$ dividing $r$. 
The class $\calN_{good}$ is, for some lines of Table~\ref{tbl:elements}, 
a union of several conjugacy classes of elements of $S_n$, say $\calN_{good}
=\cup_\calC \calN(\calC)$. For $g\in\calN(\calC)$, the proportion $p(\calC)$ 
of $k$-subsets $\ga$ of $\Om$, such that $c_k(\ga,g)\ne r_0m$ for any 
$r_0$ dividing $r$, may depend on the class $\calC$, although, as we have 
shown above,  $p(\calC)\le p$ for all $\calC$. Thus, given
$g\in\calN(\calC)$, the probability that {\sc TraceCycle}$(g)=$ \true\ 
is $(1-p(\calC))^M\geq (1-p)^M$. This implies that 
$$
\prob(  {\tracecycle}(g)= \true \mid g\in {\calN_{good}})
\ge \left(1-p\right)^M.
$$
Thus to complete the proof it is sufficient to prove that
$p\leq \frac{2}{n}$ for some upper bound $p$ of $X/\binom{n}{k}$.

Note that, by Lemma~\ref{lem:simple}(ii), $m\geq n-6\geq 150$.
Suppose first that  $4\le k\le\frac{n}{2}$.
We consider the function  $F(x) =
\left(\frac{3x}{4m}\right)^{\frac{x}{2}}=e^{\frac{x}{2}\log \frac{3x}{4m}}$
on the interval $[4,\frac{n}{2}]$.
Note that  $\frac{3x}{4m}\leq \frac{3n}{8m} < 1$ and 
${\frac{k}{2}} \le {\lceil \frac{k}{2}\rceil}$, so
$F(k)\ge \left(\frac{3k}{4m}\right)^{\lceil k/2\rceil}$, and
hence $p\leq \sqrt{8k} F(k)$.
Differentiating we have $F'(x) = F(x) \frac{1}{2}\left( 
\log(\frac{3x}{4m}) + 1\right),$ and since $F(x) > 0$ for 
$x>0$, it follows that $F(x)$ has a unique minimum 
at $\log \frac{3x}{4m}=-1$, that is, when $x=\frac{4m}{3e}$ (which
may or may not lie in the interval $[4,\frac{n}{2}]$). 
Thus the maximum of $F(x)$ on the interval $[4,\frac{n}{2}]$ occurs at one of the endpoints. We claim that 
$\max\{ F(4), F(\frac{n}{2})\} < \frac{1}{n^{3/2}}.$ It follows from a proof of this claim that $p\leq  \sqrt{8k}F(k) \le 
\sqrt{8k} \frac{1}{n^{3/2}}\leq \frac{2}{n}$, since $k\leq \frac{n}{2}$.

Since $m\geq n-6\geq 150$, we have $m^2 > 9 n^{3/2}$,
which implies that $F(4)=(\frac{3}{m})^2<\frac{1}{n^{3/2}}$. Also $\frac{3n}{8m}
\le \frac{3}{8}+\frac{6}{8m} < \frac{1}{2}$, and $n^{3/2} < 2^{n/4}$. Then,   applying Lemma~\ref{lem:ns}(a), we find  
\begin{eqnarray*}
F(\frac{n}{2}) &= &  \left(\frac{3n}{8m}\right)^{n/4} <  
\left(\frac{1}{2}\right)^{n/4} < \frac{1}{n^{3/2}} \\
\end{eqnarray*}
proving the claim for $k\geq4$.
For the remaining cases where $k=2$ or 3, note that $\omega(\gcd(m,k_0))\le 1$ for $1\le k_0\le 3$, $\sigma_{k_0}=0$ when $k_0=1$,  $n\geq 156$, and $n-m\le 6$. 
If $k=2$ then by (\ref{eq:X}),
\begin{eqnarray*}
\frac{X}{\binom{n}{2}} &\le& \frac{\binom{6}{2}}{\binom{n}{2}} + \frac{\binom{\lfloor m/2\rfloor}{1}\cdot\binom{6}{0}}{\binom{n}{2}}
\le \frac{15\cdot 2}{155}\cdot\frac{1}{n} + \frac{m}{n-1}\cdot\frac{1}{n}
< \frac{2}{n}.
\end{eqnarray*}
If $k=3$ then, again  by (\ref{eq:X}),
\begin{eqnarray*}
\frac{X}{\binom{n}{3}} &\le& \frac{\binom{6}{3}}{\binom{n}{3}} + \frac{\binom{\lfloor m/2\rfloor}{1}\binom{6}{1}}{\binom{n}{3}}+\frac{\binom{\lfloor m/2\rfloor}{1}\binom{6}{0}}{\binom{n}{3}}\\
&\le& \frac{20\cdot 6}{154\cdot 155}\cdot\frac{1}{n} + \frac{3\cdot m (6+1)}{154 (n-1)}\cdot\frac{1}{n} 
< \frac{2}{n}.
\end{eqnarray*}

\end{proof}

\section{Bounding $\calS_0$}\label{sec:S0}

Let $G,m,n,r$ be as in one of the lines of
Table~\ref{tbl:elements}, so $G$ is $A_n$ or $S_n$.
To estimate the probability of a
uniformly distributed random element $g\in G$ being in $\calS_0$ or
$\calS_1^+$,  and $\tracecycle (g) = \true$ we use
the following  result from \cite{NieP}.  Recall the definitions  of an
$s$-small  and an  $s$-large cycle  and of $v$ 
from  Notation~\ref{notation}. 
Let $i\in\{1,2,3\}$.
In the next two sections we use the following notation:

\begin{notation}\label{not:p}
\begin{enumerate}
\item For $v\ge1$ let $P(v,rm)$ denote  the  proportion  of
elements  of $S_v$  of order  dividing $rm$, and let  $P(0,rm)=1$.
\item For $v\ge1$ let $P_0(v,rm)$ denote  the proportion of  elements of $S_v$ of
  order dividing  
$rm$, all  of whose  cycles are $s$-small, and let $P_0(0,rm)=1$.
\item Let $P_1^+(v,rm)$  denote the
proportion of elements $g\in S_v$ of order dividing $rm$, and such that
$g$ has exactly one $s$-large cycle of 
length $d$, say, where in addition, $d$ satisfies ${(rn)}^s \le d < v-3{(rn)}^s$. 
\item Let $D$ denote the
set of all divisors of $rm$ which are at most $n$.
\item Let  $D_1^+(v)$ denote the set of all divisors $d$ of
$rm$ satisfying ${(rn)}^s \le d < v-3{(rn)}^s$.  
\end{enumerate}
\end{notation}

Note that $r=1$  or $r$ is a prime. Hence
the number $d(rm)$ of positive divisors of $rm$ is at most $2d(m)$, as
$d|rm$ if and only if either $d|m$ or $d = r d_0$ and $d_0|m.$

\begin{remark}\label{rem:p0}
{\rm
The following result is
essentially  \cite[Lemma 2.4]{NieP}.
Suppose that  $s$, $\delta$ and $c_\delta$ are
as in Notation~$\ref{notation}$. In particular, $s > \delta.$
In Lemma~\ref{lem:p0} we may use as $a_\delta'$ any constant such that
$a_\delta' \ge \frac{5}{4}(1 + 3\frac{c_\delta}{{(rm)}^{s-\delta}} + 
\left(\frac{c_\delta}{{(rm)}^{s-\delta}}\right)^2)$  for all 
sufficiently large values of $rm$, say $rm \ge
m_0$. These conditions hold in particular for $a_\delta',m_0$ in one of the lines of Table~\ref{tbl:x}. Note that, for the proof of Theorem~\ref{main-theorem}, we have $n\geq156$ by Lemma~\ref{lem:simple}(ii), so $rm\geq 150$, as in line 2 of Table~\ref{tbl:x}. 
}
\end{remark}

\begin{table}[ht]
\begin{center}
\begin{tabular}{p{5cm}ccl}
\toprule
$a_\delta'$ & & $m_0$  \\
\midrule
$25/4$ & & $c_\delta^{1/(s-\delta)}$ \\
$a_\delta$ as in (\ref{eq:adelta})& & $150$ \\
\bottomrule
\end{tabular}
\vspace*{.3cm}
\caption{Possible values of $a_\delta'$ for Lemma~\ref{lem:p0}}\label{tbl:x}
\end{center}
\end{table}

\begin{lemma}\label{lem:p0}
Let $m,n,r$ be as in one of the lines of Table~$\ref{tbl:elements}$. Further,
let $v\ge 16$ and $s$, $\delta$, $c_\delta$ and $a_\delta$ be as in 
Notation~$\ref{notation}$. Let $a_\delta'$ and $m_0$ be as in one of the lines of Table~$\ref{tbl:x}$ (or more generally as in Remark~$\ref{rem:p0}$) and suppose that $rm\geq m_0$. 
Then
{\openup 3pt
\begin{enumerate}
\item[(a)]
${\displaystyle P_0(v,rm) < \frac{a_\delta' d(rm)r^{2s}n^{2s}}{v^3}.}$
\item[(b)] If $3{(rn)}^s < v$ then
${\displaystyle P_0(v,rm)  \le
\frac{a_\delta' d(rm)^2 r^{2s}n^{2s}}{v(v-{(rn)}^s)^3}.}$
\item[(c)]
${\displaystyle P_1^+(v,rm)=\sum_{d\in D_1^+(v)}\frac{1}{d}P_0(v-d,rm)}$.
\end{enumerate}}
\end{lemma}

\begin{proof}
This result follows from  \cite[Lemma~2.4]{NieP} and its proof.
A direct application of  \cite[Lemma~2.4]{NieP} would require that $rm\ge v$, 
which we cannot guarantee to hold. However,
the proof of that lemma shows, without the assumption that  $rm\ge v$, that
$$P_0(v,rm) \le
\frac{d(rm){(rm)}^{2s}(1+3c_{\delta}{(rm)}^{\delta-s}+(c_{\delta}{(rm)}^{\delta-s})^2)}{v(v-1)(v-2)}$$  
whenever $v\ge 3$. Statement (a) follows from this,
since  $m\le n$,  $\delta < s$ and, for $v\ge 16$,
$v(v-1)(v-2)>\frac{4}{5}v^3$.
To prove (b) we
let $D_s$ denote the set of all divisors $d$ of $rm$ such that $d <
\min\{ v,{(rn)}^s\}$.  By \cite[Lemma~2.3(a)]{NieP}  we have
that
$P_0(v, rm) = \frac{1}{v} \sum_{d\in D_s} P_0(v-d, rm)$, where $P_0(j,m)
= 0$ for $j \le 0.$
Since, using Lemma~\ref{lem:simple}(ii), 
$v-d>3{(rn)}^s-{(rn)}^s>24$ for $d\in D_s$,  we have by (a) that
$P_0(v-d,rm) \le \frac{a_\delta' d(rm) r^{2s}n^{2s}}{(v-d)^3}$.
Thus 
$P_0(v, rm) 
\le \frac{1}{v} \sum_{d\in D_s} \frac{a_\delta' d(rm) r^{2s}n^{2s}}{(v-d)^3}.$
Since
$v-d > v - r^sn^s>0$ 
for $d \in D_s,$ 
and $|D_s| \le d(rm)$, we have
$P_0(v, rm)  \le  \frac{a_\delta' d(rm)^2 r^{2s}n^{2s}}{v(v-{(rn)}^s)^3}.$ 

Finally, we prove (c).  
The number of permutations in $S_v$ of order dividing $rm$
with exactly one $s$-large cycle
of a given length, $d$ say, where $d$ divides $rm$ and ${(rn)}^s \le d
< v - 3{(rn)}^s$
is $\binom{v}{d} (d-1)! P_0(v-d,rm) (v-d)!$.  Hence the proportion
in $S_v$ of such permutations is 
$\frac{1}{d} P_0(v-d,rm)$. Summing over all $d\in D_1^+(v)$ yields the
desired result. 
\end{proof}

\begin{proposition}\label{prop:f0}
Let $G, m, n, r$ be as in one of the lines of Table~$\ref{tbl:elements}$.
If $12{(rn)}^s  + 6\le n$ and
${(rn)}^s\log(n) \le n$ then, for a uniformly distributed random element
$g\in G$,
$$\prob( g\in \calS_0\cap G \mbox{\ and\ } \mbox{\sc TraceCycle}(g) = 
\true ) \le 
a_{\delta} d(rm)^2 r^{2s}\frac{72}{n^{3-2s}}
$$
where $a_\delta$ is as in $(\ref{eq:adelta})$.
\end{proposition}

\begin{proof}
The set $\calS_0 = \dot\cup \calS_0(v)$, where
$\calS_0(v)$ is the set of all $g\in\calS_0$ with $|\Del(g)| = v$,
where $v$ ranges over all integers satisfying  $4{(rn)}^s<v\leq n$.

For $g\in \calS_0(v)$, the restriction $g^{\Delta(g)}$
of $g$ to $\Del(g)$ is a
permutation in $\Sym(\Delta(g))$ of order dividing $rm$ with all
cycles of length less than ${(rn)}^s$.  Consider a fixed $v$-set $\Delta$.
If $G=S_n$, then all elements of $\Sym(\Delta)$ are induced by
permutations in $G$. On the other hand if $G=A_n$, then one of the lines 4-9 of
Table~\ref{tbl:elements} holds and hence  
$rm$ is odd; thus all  elements
of $\Sym(\Delta)$ of order dividing $rm$ actually 
lie in $\Alt(\Delta)$ and are therefore induced by elements of $G$.
Therefore in all cases the number of possibilities for the restriction $g^\Delta$ of
elements $g\in G$, for a
given $v$-subset $\Delta = \Delta(g)$, is $v!P_0(v,rm)$ and 
the restriction $g^\Sigma$ where $\Sigma = \Omega \backslash \Delta$
lies in ${\rm Sym}(\Sigma)$ or ${\rm Alt}(\Sigma)$ according
as $G=S_n$ or $A_n$, respectively. 
Hence
 the number
of permutations in
$\calS_0\cap G$ corresponding to this value of $v$ satisfies 
$$
|\calS_0(v)\cap G| \le \binom{n}{v} v! P_0(v,rm)
\frac{(n-v)!}{|S_n:G|} = n! \frac{P_0(v,rm)}{|S_n:G|} = |G|\cdot P_0(v,rm).
$$  
As $3{(rn)}^s < 4{(rn)}^s < v$, we have $n\geq156$ by Lemma~\ref{lem:simple}(ii) so $rm\geq150$, and hence we can 
apply Lemma~\ref{lem:p0}(b) with $a_\delta'=a_\delta$. Thus, for a random $g\in G$,
$$\prob( g\in \calS_0(v)\cap G) \le P_0(v,rm) \le 
\frac{a_\delta d(rm)^2 r^{2s}n^{2s}}{v(v-{(rn)}^s)^3}.$$
For any $g \in S_n$ with $|\Delta(g)| = v$ and $v \le n - k - 1$, 
we have in particular $3 \le v \le n-3$. Hence by
Lemma~\ref{TraceTwo}(b), given that $g\in\calS_0(v)\cap G$ with $v\leq n-k-1$,
$$\prob( \mbox{\sc TraceCycle}(g) = \true
 ) \le  16
\left(\frac{v}{n}\right)^4.$$ 
Hence, if $v \le n - k -1$, then the probability that $g\in \calS_0(v)$ and
$\tracecycle(g)=\true$ is at most
$
a_\delta d(rm)^2 r^{2s}n^{2s} \frac{16}{v(v-{(rn)}^s)^3}
\left(\frac{v}{n}\right)^{4}$; and if  $n-k-1 < v \le n$, this probability
is at most
$
a_\delta d(rm)^2 r^{2s}n^{2s} \frac{1}{v(v-{(rn)}^s)^3}.
$ 
Summing over the  values of $v$,  we find
\begin{equation*}
\prob( g\in \calS_0 \cap G\mbox{\ and\ } \mbox{\sc 
TraceCycle}(g) = \true )\le \Sigma_1 + \Sigma_2 
\end{equation*}
where
\begin{eqnarray*}
 \Sigma_1 &=& 16 a_\delta d(rm)^2 \frac{r^{2s}n^{2s}}{n^{4}}\sum_{4{(rn)}^s < v\leq n-k-1}
\frac{ v^{3}}{(v-{(rn)}^s)^3},\\
\Sigma_2 &=& 
a_\delta d(rm)^2 r^{2s}n^{2s}\sum_{n-k\le v\leq n} \frac{1}{(v-{(rn)}^s)^4}. \\
\end{eqnarray*}

We first consider $\Sigma_1$ and apply
Lemma~\ref{lem:int1} with $a=4{(rn)}^s$, $c = {(rn)}^s$, 
$t=\ell=3$, and $n-k-1$ in place of $n$. We also use $a-1-c = 3{(rn)}^s-1 > 2{(rn)}^s$, and find

\begin{eqnarray*} 
\Sigma_1 & = & 16 a_\delta d(rm)^2 \frac{r^{2s}n^{2s}}{n^{4}}\sum_{4{(rn)}^s < v\leq n-k-1} 
\frac{ v^{3}}{(v-{(rn)}^s)^3} \\ 
&<& 16 a_\delta d(rm)^2 \frac{r^{2s}n^{2s}}{n^{4}}
\left(\frac{{(rn)}^{3s}}{8 {(rn)}^{2s}} + \frac{3{(rn)}^{2s}}{2{(rn)}^s} \right.\\ 
&&
\left. +  \binom{3}{2} {(rn)}^{s} \log (n-k-1) +  \binom{3}{3} 
\frac{\left((rn)^{s}\right)^0(n-k-{(rn)}^s)^{1}}{1}\right)\\
&<& 16 \frac{a_\delta d(rm)^2 r^{2s} }{n^{3-2s}}
\left(\frac{{(rn)}^s}{8n} + \frac{3{(rn)}^s}{2n}
+  \frac{3\log(n){(rn)}^s}{n }
 +  \frac{n-{(rn)}^s}{n}
\right).\\ 
\end{eqnarray*}

The assumption $12{(rn)}^s + 6 \le n$ implies by Lemma~\ref{lem:simple}(i) that
 ${(rn)}^s/n < 1/12$.
Also, by our hypothesis, ${(rn)}^s\log(n) \le n$ and,
therefore, 
$\Sigma_1 < \frac{16 a_\delta d(rm)^2r^{2s}}{n^{3-2s}}  (\frac{1}{96} +
\frac{3}{24} + 3 + 1) 
< \frac{66.2 a_\delta d(rm)^2r^{2s}}{n^{3-2s}} .$
Finally, we estimate $\Sigma_2.$
$$\Sigma_2 = 
a_\delta d(rm)^2 r^{2s}n^{2s}\sum_{n-k \le v\leq n} \frac{1}{(v-{(rn)}^s)^4}
.$$
Since $k \le n/2,$ and since $\frac{1}{v-{(rn)}^s}$ is decreasing for
$v$ in the interval $[n-k-2, n]$, we have by Lemma~\ref{lem:inte} and
Lemma~\ref{lem:simple} that
\begin{eqnarray*}
\Sigma_2 
&<&  a_\delta d(rm)^2 r^{2s}n^{2s}\int_{n/2-1}^n \frac{1}{(v-{(rn)}^s)^4}{\rm d}v\\
&=& a_\delta d(rm)^2 r^{2s}n^{2s} \left[\frac{-1}{3(v-{(rn)}^s)^3}\right]_{n/2-1}^n\\
&<& a_\delta d(rm)^2 r^{2s}n^{2s} \frac{1}{3(n/2-1-{(rn)}^s)^3}\\
&=& a_\delta d(rm)^2 r^{2s}n^{2s} \frac{8}{3n^3(1-2/n-2{(rn)}^s/n)^3}\\
&<& a_\delta d(rm)^2 r^{2s}n^{2s} \frac{8}{3n^3(1-2/156-2/12)^3}\\
&<& 4.83 \, a_\delta d(rm)^2  r^{2s}\frac{1}{n^{3-2s}}.\\
\end{eqnarray*}

%
%

Adding the upper bounds for $\Sigma_1$ and $\Sigma_2$ 
we find that 
$$\prob( g\in \calS_0 \cap G \mbox{\ and\ } \mbox{\sc
    TraceCycle}(g) = \true) 
<a_\delta d(rm)^2 r^{2s}\frac{72}{n^{3-2s}}.$$
\end{proof}

\section{Bounding $\calS_1^+$}\label{sec:S1+}

Let $G,m,n,r$ be as in one of the lines of
Table~\ref{tbl:elements}, so $G$ is $A_n$ or $S_n$.

Recall the definitions  of an
$s$-small  and an  $s$-large cycle  and of $v$ 
from  Notation~\ref{notation} and the  notation set out in
Notation~\ref{not:p}.

\begin{proposition}\label{prop:f1}
Let $G, n, m, r$ be as in one of the lines of Table~$\ref{tbl:elements}$.
If  $n$ is such that $12{(rn)}^s + 6 \le n$ and
${(rn)}^s\log(n) \le n$ , then for a uniformly distributed random element
$g \in G$,
$$\prob( g\in \calS_1^+ \cap G\mbox{\ and\ } \mbox{\sc TraceCycle}(g) = 
\true ) \le 
a_\delta d(rm)^3\frac{6.24}{n^{1+s}}.$$
\end{proposition}

\begin{proof} 
The set $\calS_1^+ = \dot\cup \calS_1^+(v)$, where
$\calS_1^+(v)$ is the set of all
$g\in\calS_1^+$ with $|\Del(g)| = v$ and $v$ ranges over integers satisfying
$4(rn)^s<v\leq n$.
For a given $v$, an analogous argument to that given in the second paragraph of the proof of Proposition~\ref{prop:f0} shows that 
\begin{eqnarray*}
|\calS_1^+(v)\cap G| &\le& \binom{n}{v}\cdot v! P_1^+(v,rm)\cdot
\frac{(n-v)!}{|S_n:G|}\\ 
&= &n! \frac{P_1^+(v,rm)}{|S_n:G|}
= P_1^+(v,rm)\cdot |G|.
\end{eqnarray*}
Thus applying Lemma~\ref{lem:p0}(c) we have, for a random $g\in G$, 
\begin{eqnarray*}
\prob( g\in \calS_1^+(v)\cap G) 
&\le& P_1^+(v,rm)
= \sum_{d\in D_1^+(v)}\frac{1}{d}P_0(v-d,rm).
\end{eqnarray*}
If $|\Delta(g)| = v$
and $v \le n - k -1$, then in particular  $3\leq v\leq n-3$.
Hence by Lemma~\ref{TraceTwo}(b), given that $g\in \calS_1^+(v)\cap G$ 
with $|\Delta(g)| = v$ with $v \le n - k -1, $  
$$
\prob( \mbox{\sc TraceCycle}(g) = \true ) \le  16
\left(\frac{v}{n}\right)^4.
$$
Thus, if $v \le n-k-1$, the probability that $g\in \calS_1^+(v)\cap G$ and 
$\mbox{\sc TraceCycle}(g) = \true$ is at most
$$
\left( \sum_{d\in D_1^+(v)}\frac{1}{d}P_0(v-d,rm)\right)
16
\left(\frac{v}{n}\right)^4,$$
and if $n-k \le v$ this probability is at most
$$
 \sum_{d\in D_1^+(v)}\frac{1}{d}P_0(v-d,rm).
$$
Summing over $v$ we find
\[
\prob( g\in \calS_1^+ \cap G\mbox{\ and\ } \mbox{\sc 
TraceCycle}(g) = \true )\leq \Sigma_1 + \Sigma_2
\]
where
\begin{eqnarray*}
\Sigma_1 &=& 
16 \sum_{4{(rn)}^s< v\leq n-k-1}\left(\sum_{d\in
     D_1^+(v)}\frac{1}{d}P_0(v-d,rm)\right) 
\left(\frac{v}{n}\right)^4,\\
\Sigma_2 &=& \sum_{n-k\le v\leq n}\left(\sum_{d\in
     D_1^+(v)}\frac{1}{d}P_0(v-d,rm)\right).
\end{eqnarray*}

First  we consider  $\Sigma_1$. Interchanging  the two  summations and
taking the sum  up to $n$, we obtain the  following upper bound, where
$D_\ell$ denotes the set of all divisors $d$ of $rm$ satisfying $d \ge
{(rn)}^s.$ Note that $v > d + 3{(rn)}^s$ (see Notation~\ref{notation}).

\begin{eqnarray*}
\Sigma_1 < 16\sum_{d\in D_\ell} \frac{1}{d}\left(\sum_{3{(rn)}^s+d
  < v\le n} P_0(v-d, rm) \cdot \frac{v^{4}}{n^{4}}\right).\\
\end{eqnarray*} 

Since $rm\geq150$ by Lemma~\ref{lem:simple}(ii), we may apply 
Lemma~\ref{lem:p0}(b) with $a_\delta'=a_\delta$, and find that
 this expression is at most
\begin{eqnarray*}
\lefteqn{16\sum_{d\in D_\ell} \frac{1}{d}\left( \sum_{3{(rn)}^s+d < v\le n}
\frac{a_\delta d(rm)^2 r^{2s}n^{2s}}{(v-d)(v-d-{(rn)}^s)^3} \cdot
\frac{v^{4}}{n^{4}}\right)}\\ 
&<& 16\frac{a_\delta d(rm)^2 r^{2s}n^{2s}}{n^{4}}
\sum_{d\in D_\ell} \frac{1}{d} \left(\sum_{3{(rn)}^s+d < v\le n}
\frac{v^{4}}{(v-d-{(rn)}^s)^4}\right).\\
\end{eqnarray*}
Now we apply Lemma~\ref{lem:int1} with $t =  \ell = 4,\, a = 3{(rn)}^s + d$ and
$c = d+{(rn)}^s$. Noting that $a-c-1 = 2{(rn)}^s-1$, we obtain that this
expression is  at most
\begin{eqnarray*}
\lefteqn{16\frac{a_\delta d(rm)^2r^{2s}n^{2s}}{n^{4}}
\sum_{d\in D_\ell}\frac{1}{d} \left(
\frac{\binom{4}{0}(d+{(rn)}^s)^{4}}{3(2{(rn)}^s-1)^3} +
\frac{\binom{4}{1}(d+{(rn)}^s)^{3}}{2(2{(rn)}^s-1)^2}\right.}\\ 
&&  + \frac{\binom{4}{2}(d+{(rn)}^s)^{2}}{(2{(rn)}^s-1)} +
\binom{4}{3}(d+{(rn)}^s)^{1} \log(n)\\ 
&& \left.+ \binom{4}{4} (d+{(rn)}^s)^0(n+1-d-{(rn)}^s)^{1} 
\right).\\
\end{eqnarray*} 
Note that 
$2{(rn)}^s -1 > \frac{23}{12}r^sn^s$  by Lemma~\ref{lem:simple}(iii) 
and, since
$d \ge {(rn)}^s$, $\,$ also
$\frac{d + {(rn)}^s}{d}  \le 2.$ 
Note also that $d + (rn)^s < n$ and
$n+1-d-(rn)^s  < n$.
Hence 
\begin{eqnarray*}
\Sigma_1 &\le & 16\frac{a_\delta d(rm)^3r^{2s}n^{2s}}{n^{4}}
\left(\frac{2\cdot 12^3\cdot n^{3}}{3\cdot 23^3\cdot r^{3s}n^{3s}}
  + \frac{4\cdot 12^2 n^{2}}{23^2 r^{2s}n^{2s}}
  + \frac{12\cdot 12 n^{1}}{23 r^sn^s} \right.  \\
  &+& \left. 8\cdot \log (n) 
  + \frac{n^{1}}{{(rn)}^s}  \right)\\
&=&
\frac{16a_\delta d(rm)^3}{n^{1+s}}\left(
\frac{3456}{36501 r^{s}} + \frac{576}{529 n^{1-s}}
+ \frac{144 r^s}{23n^{2-2s}} \right. \\
&+& \left. \frac{8r^{2s}\log(n)}{n^{3-3s}} 
+ \frac{r^s}{n^{2-2s}} \right).\\
\end{eqnarray*}
Since, by hypothesis ${(rn)}^s\log(n)\le n$  and  by
Lemma~\ref{lem:simple}(i) $n^s/n \le r^sn^s/n  \le 1/12$ and
$r \ge 1$,
the last expression is at most
\begin{eqnarray*}
\frac{16a_\delta d(rm)^3}{n^{1+s}}&&\left(
\frac{3456}{36501} + \frac{576}{529\cdot 12}
+ \frac{144}{23 \cdot 12^2} 
+ \frac{8}{12^2} 
+ \frac{1}{12^2} 
 \right) \\
&\le & 4.7 a_\delta\frac{d(rm)^3}{n^{1+s}}.\\
\end{eqnarray*}

We now consider 
${\Sigma_2 = \sum_{n-k\le v\leq n}\left(\sum_{d\in 
     D_1^+(v)}\frac{1}{d}P_0(v-d,rm)\right).}$
As
$v-d > 3{(rn)}^s$ and   $n-k \ge n/2$  we have by
 Lemma~\ref{lem:p0}(b) (with $a_\delta'=a_\delta$) that
\begin{eqnarray*}
\frac{\Sigma_2}{a_\delta d(rm)^2 {(rn)}^{2s}}
&\le&  \sum_{n/2\le v\le n}
\left(\sum_{d\in      D_1^+(v)}
\frac{1}{d(v-d-{(rn)}^s)^4} \right)\\
&=& \sum_{d\in      D_1^+(v)}
\frac{1}{d}\left(\sum_{v(d)\leq v\leq n}\frac{1}{(v-d-{(rn)}^s)^4} \right) 
\end{eqnarray*}
where $v(d)=\max\{\frac{n}{2}, d+3(rn)^s\}$ since, by Notation~\ref{not:p}, 
each $d\in D_1^+(v)$ is less than
$v-3(rn)^s$. 
By Lemma~\ref{lem:inte}, this quantity is at most
\begin{eqnarray*}
\lefteqn{ \sum_{d\in      D_1^+(v)}
\frac{1}{d}\left(\int_{v(d)-1}^n
\frac{1}{(v-d-{(rn)}^s)^4}\,{\rm d}v\right)}\\
&=& \sum_{d\in      D_1^+(v)}
\frac{1}{d} \left[-\frac{1}{3}\frac{1}{(v-d-(rn)^s)^3}\right]_{v(d)-1}^n \\
&<&  \sum_{d\in      D_1^+(v)}
\frac{1}{3d}\,  \frac{1}{(v(d)-1-d-(rn)^s)^3}.\\
\end{eqnarray*}
In particular each $d\in D_1^+(v)$ is less than $m.$
By Lemma~\ref{lem:d}, there are at most three divisors 
of $rm$ which are less than $m$ and greater than $2m/7$, and the sum
of the reciprocals $\tfrac{1}{d}$ of these divisors is at most $\tfrac{7}{m}$,
which is less than $\tfrac{7.3}{n}$ since $n\ge 156$ (by Lemma~\ref{lem:simple}(ii)).
Using $v(d)\geq d+3(rn)^s$ and Lemma~\ref{lem:simple}(iii), 
the contribution from these exceptional divisors is therefore at most
$$
\frac{1}{(2(rn)^s-1)^3} \sum_{d\in D_1^+(v), d> 2m/7}\frac{1}{3d}< 
\left(\frac{12}{23(rn)^s}\right)^3\, \frac{7.3}{3n} < \frac{0.35}{(rn)^{3s}n}.
$$
Finally we estimate the contribution of the remaining elements $d$ of $D_1^+(v)$. We note that each such $d$ is at most $\tfrac{2n}{7}$ and at least $(rn)^s$, and that 
${(rn)}^s < \tfrac{n-6}{12}$ by our hypothesis. Thus, using 
$v(d)\ge \frac{n}{2}$, 
the remaining contribution is at most 
$$
\frac{d(rm)}{3{(rn)}^s}
\frac{1}{(\tfrac{n}{2}-1-\tfrac{2n}{7}-\tfrac{n-6}{12})^3}.
$$
Observe that 
$\tfrac{n}{2}-1-\tfrac{2n}{7}-\tfrac{n-6}{12}= 
\tfrac{11n - 42}{84}$ and since $n > 84$  by
Lemma~\ref{lem:simple}(a) we have 
$\tfrac{11n - 42}{84} > \tfrac{n}{8}.$ Hence, using also that $\tfrac{(rn)^s}{n}<\tfrac{1}{12}$ (by Lemma~\ref{lem:simple}(i)), the above expression is less than 
$$
\frac{d(rm)}{(rn)^s}\frac{8^3}{3n^3} < \frac{d(rm)\, 8^3}{12^2\cdot 3 \,(rn)^{3s}n}
< \frac{1.19\,d(rm)}{(rn)^{3s}n}.
$$
Thus
\begin{eqnarray*}
\frac{\Sigma_2}{ a_\delta d(rm)^2(rn)^{2s}} &<& 
\frac{0.35}{(rn)^{3s}n} + \frac{1.19\,d(rm)}{(rn)^{3s}n} \leq \frac{1.54\,d(rm)}{(rn)^{2s}n^{1+s}}
\end{eqnarray*}
and hence
\begin{eqnarray*}
\prob( g\in \calS_1^+ \cap G\mbox{\ and\ }
\mbox{\sc TraceCycle}(g) = \true) 
& < &
6.24\, a_\delta\frac{d(rm)^3}{n^{1+s}}.\\
\end{eqnarray*}
\end{proof}

\section{Bounding $\calS_{\ge 2}$}\label{sec:S2}

\begin{proposition}
\label{prop:Sge2}
Let $G,m,n,r$ be as in one of the lines of Table~$\ref{tbl:elements}$. 
Then
\[
\frac{\left|\calS_{\ge 2}\cap G\right|}{|G|}
\le \frac{d(rm)^2}{{(rn)}^{2s}}.
\]
\end{proposition}
\begin{proof}
If $g$ is an element of $\calS_{\ge 2}\cap G$ then it  has two cycles of
lengths $d_1, d_2$, where $d_i|rm$, and $d_i\ge {(rn)}^s$. There are at most
$d(rm)$ choices for each $d_i$.
Thus, there are at most $d(rm)^2$ choices for the two 
divisors
$d_1$ and $d_2$. For a given $d_1, d_2$, the proportion of elements in
$G$ having cycles of lengths $d_1$ and $d_2$ is at most
\[
{(d_1d_2)}^{-1}\le {{(rn)}^{-2s}}.
\]
Thus altogether we get a proportion of at most $d(rm)^2 
{{(rn)}^{-2s}}$.
\end{proof}

\section{Bounding $\calS_1^-$}\label{f-1}

\begin{proposition}\label{prop:f1bound}
Let $G,m,n,r$ be as in one of the lines of Table~$\ref{tbl:elements}$. 
Suppose that $n$ is such that $12{(rn)}^s + 6 \le n.$  Let $k$ be a fixed
integer with $2 \le k \le n/2.$
  Then
\begin{enumerate}
\item[(a)]the proportion of $k$-subsets $\ga$ such that
  $c_k(\ga,g)=r_0m$, for some $r_0$ dividing $r$, 
for $g \in \calS_1^{-}\cap G$, is less than $31/({(rn)}^{1-s})$.
\item[(b)] If \mbox{\sc TraceCycle} is
   Algorithm~{\rm\ref{algo:TraceCycle}} and $M$ is as
defined there, then for a uniformly distributed random element $g\in G$,
$$\prob(\mbox{\sc TraceCycle}(g) =
\true \mid g\in\calS_1^-\cap G) <
\left(\frac{31}{{(rn)}^{1-s}}\right)^M,$$
and so
$$\prob(g \in \calS_1^-\cap G \mbox{\ and\ } \tracecycle(g) = \true )<
\displaystyle{\left(\frac{31}{{(rn)}^{1-s}}\right)^M}.$$
\end{enumerate}
\end{proposition}

\begin{proof}

We start by recording some important facts used throughout the proof.
Let $g\in\calS_1^-\cap G$ and put 
$v=|\Del(g)|$ and 
$u=|\Sigma(g)|,$ such that $u + v = n.$ The definition of 
$\calS_1^-$ implies that $g$ has a unique $s$-large cycle $C$ in
$\Del(g)$ of length $d$  and we have
\begin{itemize}
\item[(i)] $d \le n$ and $d\not=m$ since $g\in \calF;$
\item[(ii)]  $v> 4{(rn)}^s$ and $v - d \le 3{(rn)}^s.$ 
\end{itemize}
By Lemma~\ref{lem:d}
and the hypothesis
 $n \ge 12{(rn)}^s + 6,$ 
 it follows that $d\le 2m/3 \le 2n/3.$
Hence $u = n-v \ge n - d - 3{(rn)}^s \ge \frac{n}{3} - 
3{(rn)}^s \ge 4{(rn)}^s + 2 -3{(rn)}^s  = 
{(rn)}^s + 2.$ Also, 
$v\le d + 3{(rn)}^s \le \frac{2n}{3} + 3{(rn)}^s .$ 
This implies that
$v = n - u \le n - 2 -{(rn)}^s $ and hence in particular
\begin{equation}\label{eq:v}
v  \le n - 3
\end{equation}
and 
\begin{equation}\label{eq:u}
\frac{1}{u-1} < \frac{1}{{(rn)}^s } <\frac{1}{{(rn)}^{1-s}}.
\end{equation}

Set $t=v-d$ so that $t=v-d\leq 3{(rn)}^s$.   Then 
\begin{equation}\label{eq:v2}
v = d + t \le  2n/3 + 3{(rn)}^s.
\end{equation}

Suppose that $\ga$
is a $k$-subset for which $c_k(\ga,g)=r_0m$, for some $r_0$ dividing $r$,
 and set $k_0:=|\ga\cap \Sigma(g)|$. 
Then $c_{k_0}(\ga\cap\Sigma(g),g)$ divides $rm$, and hence the number of possibilities for the $k_0$-subset $\ga\cap\Sigma(g)$ is at most the number
$\sigma(k_0,\Sigma(g))$ of Corollary~\ref{pc}. 
In particular $\sigma(k_0,\Sigma(g)) = 0$ if  $k_0=1$. Thus 
$k_0=0$ or 
$2\le k_0\le \min\{u,k\},$  and the case $k_0=0$ is only
possible if $v \ge k.$ 

First we prove the following upper bound for 
the number $K_{\neg 0}= 
K_{\neg 0}(g)$ of $k$-subsets $\gamma$ such that 
$k_0=|\gamma \cap \Sigma(g)| \ge 2.$
\begin{eqnarray} 
\frac{K_{\neg 0}}{\binom{n}{k}} 
< \frac{97}{96{(rn)}^{1-s}}.\label{eq:parttwo}
\end{eqnarray}
By the remarks above
\[
K_{\neg 0} \leq \sum_{k_0=2}^{\min\{k,u\}}  \sigma(k_0,\Sigma(g))
\binom{n-u}{k-k_0}.
\]
If $k_0 \le u-1$ then, by Corollary~\ref{pc} and our considerations above,
$\sigma(k_0,\Sigma(g)) \le \frac{1}{u-1} \binom{u}{k_0}
\le \frac{1}{{(rn)}^{(1-s)}}\binom{u}{k_0}$, while if $k_0 = u$
then $\sigma(k_0,\Sigma(g))=1.$
Thus
\begin{eqnarray*}
\frac{K_{\neg 0}}{\binom{n}{k}} &\le&
\frac{1}{\binom{n}{k}(rn)^{1-s}}\sum_{k_0=2}^{\min\{u-1,k\}}
\binom{u}{k_0} \binom{n-u}{k-k_0} + R_{\neg 0}, 
\end{eqnarray*}
where 
$$R_{\neg 0} =
\begin{cases}
0 & \mbox{\ if\ } k \le u-1, \\
\frac{\binom{n-u}{k-u}}{\binom{n}{k}}& \mbox{\ if\ } k \ge u.
\end{cases}
$$ 
Hence 
\begin{eqnarray}
\frac{K_{\neg 0}}{\binom{n}{k}} 
&\le&
\frac{1}{\binom{n}{k}\,(rn)^{1-s}}\sum_{k_0=0}^{\min\{u,k\}}
\binom{u}{k_0}\binom{n-u}{k-k_0}  + R_{\neg 0}
\nonumber\\ 
& = &  \frac{1}{{(rn)}^{1-s}} + R_{\neg 0}.
\label{eq:partone}
\end{eqnarray}
Thus $(\ref{eq:parttwo})$ is proved if 
$k \le u-1$, so
suppose that $k \ge u.$
Recall that $u > {(rn)}^{1-s}+1$ by~(\ref{eq:u}).  Hence
$$
R_{\neg 0} \le \frac{\binom{n-u}{k-u}}{\binom{n}{k}} 
= \prod_{i=1}^u\frac{(k-u+i)}{(n-u+i)} \le
\left(\frac{k}{n}\right)^u \le
\left(\frac{1}{2}\right)^u \le
\frac{1}{2}\left(\frac{1}{2}\right)^{{(rn)}^{(1-s)}}.
$$
Using $n^{1-s} > 12>8$ (see Lemma~\ref{lem:simple}(i))
and Lemma~\ref{lem:ns}(a), 
we have $R_{\neg 0} \le \frac{1}{2} \left(\frac{1}{2}\right)^{{(rn)}^{(1-s)}}
< \frac{1}{2} \frac{1}{4 {(rn)}^{2(1-s)}} 
< \frac{1}{96} \frac{1}{{(rn)}^{1-s}}$,
and now the inequality $(\ref{eq:parttwo})$ follows from
inequality $(\ref{eq:partone})$.

To complete the proof of part (a) it remains to estimate  
the number $K_{=0}= K_{=0}(g)$ of $k$-subsets $\ga\subseteq
\Del(g)$ such that  
$c_k(\ga,g) =r_0m$ for some $r_0$ dividing $r$. 
Since this number is zero if $v<k$, we assume that $v \ge k$.
Recall that $C$ is the unique $s$-large cycle of $g$ contained in
$\Delta(g)$ and $d = |C|.$ 
By Lemma~\ref{lem:d},  $d\le2m/3<2n/3$. Since $m$ divides 
$c_k(\ga,g)$ it follows that $\ga\not\subseteq C$. We prove

\begin{eqnarray}
\frac{K_{=0}}{\binom{n}{k}} \le
\frac{30.6} {{(rn)}^{1-s}}. \label{eq:partthree}
\end{eqnarray}
The number $K_{=0}$ of such
$k$-subsets is at most $\binom{v}{k} - \binom{d}{k}$. 

Set $t=v-d$ so that
$t=v-d\leq 3{(rn)}^s$.   
Then
we have
$$\binom{v}{k} = \frac{1}{k!} (d+t)(d+t-1)\dots (d+t-k+1).$$

We  consider separately the cases 
(i) ${(rn)}^s < k$,\ (ii) $k \le \min\{{(rn)}^s, d-t+1\}$,\,  and
(iii) $d-t+1 < k \le {(rn)}^s$.
Recall that $(rn)^s \le d.$

Consider first Case (ii), so $k \le {(rn)}^s$ and
$d - t +1 \ge k.$ If $d \le m/2$ define $a=1$ and
observe that $\frac{t}{d-k+1} \le a$. 
If $d > m/2$ then, by Lemma~\ref{lem:d}, it follows that $d \ge 3m/5.$
In this case $\frac{t}{d-k+1} \le 
\frac{3 (rn)^s}{3m/5 - {(rn)}^s} =  
\frac{3 (rn)^s}{m(3/5 - {(rn)}^s/m)}.$  
By the hypothesis ${(rn)}^s \le (n-6)/12 \le m/12$ and by
Lemma~\ref{lem:simple}(i) we have then
$\frac{t}{d-k+1} \le  
\frac{ 3}{12} \frac{1}{(3/5 - 1/12)} = \frac{15}{31}.$
In this case define $a =   \frac{15}{31}.$ Then again
$\frac{t}{d-k+1} \le  a.$
Setting $d_{(k)}:=d(d-1)\dots(d-k+1)$, 
by Lemma~\ref{lem:ZZ} we obtain
\begin{eqnarray}
\binom{v}{k} &=& \frac{1}{k!} (d+t)(d+t-1)\dots (d+t-k+1)\nonumber \\
&<&\frac{1}{k!}\left(d_{(k)}\left(1+\frac{(1+a)^kt}{a(d-k+1)}\right)\right)\nonumber \\
&=&\binom{d}{k}\left( 1+\frac{(1+a)^kt}{a(d-k+1)}\right).\label{eq:vk}
\end{eqnarray}
If  $d\leq m/2$ we have $a=1$ and so
\begin{eqnarray*}
\binom{v}{k} 
&\leq&\binom{d}{k} + \binom{d}{k}
\frac{2^kt}{d-k+1}.
\end{eqnarray*}
Applying Lemma~\ref{lem:Z}(a) with $\alpha=\tfrac{1}{2}$,
\[
\binom{d}{k}\leq 
\frac{1}{2^{k-1}}\binom{n}{k}\frac{d-k+1}{n-k+1}.
\]
Hence
\[
K_{=0}\leq \binom{v}{k}-\binom{d}{k}<\binom{d}{k}\frac{2^kt}{d-k+1}\leq 
\binom{n}{k}
\frac{2t}{n-k+1} <
\binom{n}{k}
\frac{2t}{n-k}.
\]
On the other hand, if $m/2 < d\leq 2n/3$, then $a =   \frac{15}{31},$ and (\ref{eq:vk}) becomes
\begin{eqnarray*}
\binom{v}{k}
&\leq&\binom{d}{k}+ \binom{d}{k}{\left(\frac{46}{31}\right)}^k\frac{31t}{15(d-k+1)}.\\
\end{eqnarray*}
By Lemma~\ref{lem:Z}(a) with $\alpha=\tfrac{2}{3}$,
\[
\binom{d}{k}\leq 
\frac{2^{k-1}}{3^{k-1}}\binom{n}{k}\frac{d-k+1}{n-k+1}
\]
and hence
\begin{eqnarray*}
K_{=0}&=& \binom{v}{k}-\binom{d}{k}
<\binom{d}{k}\frac{{(46/31)}^k31t}{15(d-k+1)}\\
& < &
\binom{n}{k}
\frac{2^{k-1}}{3^{k-1}}\frac{{(46/31)}^k 31t}{15(n-k+1)} \\
& < &
\binom{n}{k}
\frac{92^{k}}{93^{k}}\frac{31t}{10(n-k)}  <
\binom{n}{k}
\frac{31t}{10(n-k)}.
\end{eqnarray*}
Note that by Lemma~\ref{lem:simple}, since $k \le {(rn)}^s$ and by
our assumptions, $\frac{t}{n-k} \le  \frac{3(rn)^s}{n(1-k/n)} 
\le \frac{3{(rn)}^s}{n(1-{(rn)}^s/n)}
\le \frac{3{(rn)}^s}{n(11/12)} 
=   \frac{36{(rn)}^s}{11n}.$ 

Thus for all $d$ we have
\begin{eqnarray*}
\frac{K_{=0}}{\binom{n}{k}}& < & \frac{31\cdot 36}{10\cdot 11}  \cdot \frac{ {(rn)}^s}{n} <
\frac{ 10.2\,r}{{(rn)}^{1-s}}\leq \frac{ 30.6}{{(rn)}^{1-s}} 
\end{eqnarray*}
and (\ref{eq:partthree}) is proved for Case (ii).

Now consider Cases (i) and (iii).
Recall  from~(\ref{eq:v2}) that $v = d + t \le 
 2n/3 + 3{(rn)}^s$. By Lemma~\ref{lem:simple}(i),
${(rn)}^s \le \frac{1}{12}n.$  Therefore $v \le 2n/3 + \frac{3}{12}n = 
\frac{11}{12}n.$  This shows, using Lemma~\ref{lem:Z}(a), that 
\begin{eqnarray*}
\frac{K_{=0}}{\binom{n}{k}} &\leq &
 \frac{\binom{v}{k}-\binom{d}{k}}{\binom{n}{k}}
<\frac{\binom{v}{k}}{\binom{n}{k}} \\
& \leq & \left(\frac{v}{n} \right)^k
 \leq  \left(\frac{11}{12} \right)^k. \\
\end{eqnarray*}

In Case (i) we have $k > {(rn)}^s$ and hence, observing that ${(rn)}^s >
 n^{1/2} > 12$ by Lemma~\ref{lem:simple}(i),
and using Lemma~\ref{lem:ns}(b), we have
$\left(\frac{11}{12} \right)^k < \left(\frac{11}{12} \right)^{{(rn)}^s}
 < \frac{5}{{(rn)}^s} < \frac{5}{{(rn)}^{1-s}}.$  Thus
$\frac{K_{=0}}{\binom{n}{k}}< \frac{5}{{(rn)}^{1-s}}$ and (\ref{eq:partthree}) holds
 for Case (i).

In Case (iii) 
we have ${(rn)}^s \ge k > d-t+1$ and so
$d < 4{(rn)}^s$ as $t \le 3{(rn)}^s.$ Therefore, $v = d + t < 7{(rn)}^s$, and
using Lemmas~\ref{lem:simple}(i) and~\ref{lem:Z}(a),
\begin{eqnarray*}
\frac{K_{=0}}{\binom{n}{k}} &\le & 
  \left( \frac{v}{n}\right)^k 
 <  \left( \frac{7{(rn)}^s}{n}\right)^k \\
& \le&    \left(\frac{7{(rn)}^s}{n}\right)^2
 <   \frac{49  {(rn)}^s}{12 n} \leq  \frac{49 }{{4(rn)}^{1-s}}.
\end{eqnarray*}
Thus $(\ref{eq:partthree})$ holds for Case (iii) and hence in all cases.

Combining (\ref{eq:partthree})  with
$(\ref{eq:parttwo})$, we
conclude that the proportion of $k$-subsets $\ga$
such that $c_k(\ga,g)=r_0m$, for some $r_0$ dividing $r$, is less than $31/( {(rn)}^{1-s})$ for all
values of $k$ and $v$. This proves  (a).

Now $\tracecycle(g)=\true$ if and only if $c_k(\ga,g)=r_0m$, for some $r_0$ dividing $r$, for each of the $M$
independent uniformly distributed random $k$-sets $\ga$ tested in the
algorithm. Thus, given $g\in\calS_1^-\cap G$,  the  probability that {\sc TraceCycle}$(g)= \true$ 
 is at most $\left(31/({(rn)}^{1-s})\right)^M$. 

The last assertion follows on noting 
 that for events $A$ and $B$ we have $\prob(A\cap B) =
\prob(A) \prob(B\mid A) \le \prob(B\mid A)$.
\end{proof}

\section*{Acknowledgements}

The first  author acknowledges the  support of EPSRC  grant EP/C523229
and  the  second and  third  authors  acknowledge  the support  of  ARC
Discovery  grants  DP0879134 and  DP140100416.   
 We thank Yohei Negi and Sven Reichard for some discussions
on an  early draft of  this paper. We  thank an anonymous  referee for
valuable suggestions.

\end{document}